\documentclass{amsart}

\newtheorem{theorem}{Theorem}[section]
\newtheorem{lemma}[theorem]{Lemma}
\newtheorem{problem}[theorem]{Problem}
\newtheorem{corollary}[theorem]{Corollary}
\theoremstyle{definition}

\theoremstyle{remark}
\newtheorem{remark}[theorem]{Remark}

\numberwithin{equation}{section}

\usepackage{amsmath,amssymb,amsfonts,amsthm}
\usepackage{enumerate}
\usepackage{hyperref}
\usepackage{bm}
\usepackage[hiresbb]{xcolor}%graphicx,
\usepackage{tcolorbox}
\usepackage{booktabs,tabularx}
\usepackage{mathtools}
\mathtoolsset{showonlyrefs=true}
\usepackage{mathrsfs}
\usepackage{float}
\usepackage{mathrsfs}
\usepackage[top=35truemm,bottom=35truemm,left=30truemm,right=30truemm]{geometry}

\makeatletter
\newcommand{\tpmod}[1]{{\@displayfalse\pmod{#1}}}
\makeatother
\def\rnum#1{\expandafter{\romannumeral #1}} 
\def\Rnum#1{\uppercase\expandafter{\romannumeral #1}} 
\newcommand{\1}{{\mathbf{1}}}
\newcommand{\tr}{\mathrm{tr}}
\newcommand{\Li}{\mathrm{Li}}
\newcommand{\ord}{\mathrm{ord}}
\newcommand{\rank}{\mathrm{rank}}
\newcommand{\Res}{\mathop{\mathrm{Res}}}

\newcommand{\C}{{\mathbb C}}
\newcommand{\R}{{\mathbb R}}
\newcommand{\Z}{{\mathbb Z}}
\newcommand{\Q}{{\mathbb Q}}
\newcommand{\N}{{\mathbb N}}
\newcommand{\F}{{\mathbb F}}

\makeatletter
\newcommand{\biggg}[1]{{\hbox{$\left#1\vbox to 20.5pt{}\right.\n@space$}}}
\newcommand{\Biggg}[1]{{\hbox{$\left#1\vbox to 23.5pt{}\right.\n@space$}}}
\newcommand{\bigggg}[1]{{\hbox{$\left#1\vbox to 26.5pt{}\right.\n@space$}}}
\newcommand{\Bigggg}[1]{{\hbox{$\left#1\vbox to 29.5pt{}\right.\n@space$}}}
\newcommand{\biggggg}[1]{{\hbox{$\left#1\vbox to 32.5pt{}\right.\n@space$}}}
\newcommand{\Biggggg}[1]{{\hbox{$\left#1\vbox to 35.5pt{}\right.\n@space$}}}
\newcommand{\bigggggg}[1]{{\hbox{$\left#1\vbox to 38.5pt{}\right.\n@space$}}}
\newcommand{\Bigggggg}[1]{{\hbox{$\left#1\vbox to 41.5pt{}\right.\n@space$}}}
\makeatother

\begin{document}

\title{Euler products of Selberg zeta functions in the critical strip}

%    Remove any unused author tags.

%    author one information
\author{Ikuya Kaneko}
\address{Tsukuba Kaisei High School, 3315-10 Kashiwada, Ushiku, 300-1211 Japan}
\curraddr{}
\email{ikuyak@icloud.com}
\thanks{}

%    author two information
\author{Shin-ya Koyama}
\address{Toyo University, 2100 Kujirai, Kawagoe, Saitama 350-8585 Japan}
\curraddr{}
\email{koyama@tmtv.ne.jp}
\thanks{}

\subjclass[2018]{Primary 11M36; Secondary 11M26 \and 11M06}

\keywords{Selberg zeta functions \and Partial Euler products \and Deep Riemann Hypothesis \and Explicit formula \and Prime geodesic theorem \and Selberg's eigenvalue conjecture}

\date{}

\dedicatory{}

\begin{abstract}
%We examine the asymptotic behavior of partial Euler products of Selberg zeta functions attached with finite-dimensional representations for arithmetic groups. We make certain that the precision of results hinges on advances in the prime geodesic theorem and Selberg's eigenvalue conjecture. We have evidences for these topics consulting a celebrated work of Ramanujan.
For any congruence subgroup of the modular group, we extend the region of convergence of the Euler products of the Selberg zeta functions beyond the boundary $\Re s = 1$, if they are attached with a nontrivial irreducible unitary representation. The region is determined by the size of the lowest eigenvalue of the Laplacian, and it extends to $\Re s \geqslant 3/4$ under Selberg's eigenvalue conjecture. More generally, for any unitary representation we establish the relation between the behavior of partial Euler products in the critical strip and the estimate of the error term in the prime geodesic theorem. For the trivial representation, the proof essentially exploits the idea of the celebrated work of Ramanujan.
\end{abstract}

\maketitle
\tableofcontents

\section{Introduction}%%%%%%%%%%%%%%%%%%%%%%%%%%%%%%%%%%%%%%%%%%%%%%%%%%%%%
In 1874, Mertens~\cite{Mertens1874} studied \textit{partial Euler products} for the classical Riemann zeta function $\zeta(s)$ and the Dirichlet $L$-function $L(s, \chi_{4})$ at $s = 1$, where $\chi_{4}$ is the primitive character modulo 4 (cf.~\cite{Rosen1999}).
Ramanujan~\cite{Ramanujan1997} later extended, beyond the boundary, Mertens' fantastic result to $s > 0$ in the process of obtaining the maximal order of the divisor function. Notably, if the Riemann Hypothesis (RH) for $\zeta(s)$ is assumed, then we have
\begin{equation}\label{Ramanujan}
\prod_{p \leqslant x} (1-p^{-s})^{-1}
 = -\zeta(s) \exp \left(\Li(\vartheta(x)^{1-s})+\frac{2s x^{\frac{1}{2}-s}}{(2s-1) \log x}
 + \frac{S_{s}(x)}{\log x}+O \left(\frac{x^{\frac{1}{2}-s}}{(\log x)^{2}} \right) \right)
\end{equation}
in $1/2 < s < 1$, where $\vartheta(x) = \sum \limits_{p \leqslant x} \log p$ is Tchebyshev's function and
\begin{equation}\label{zeros}
S_{s}(x) = -s \sum_{\zeta(\rho) = 0, \, \Re \rho = \frac{1}{2}} \frac{x^{\rho-s}}{\rho(\rho-s)}.
\end{equation}
This is a rather close approximation and the sum~\eqref{zeros} stems from the \textit{explicit formula} for $\vartheta(x)$.
%Despite the essence of the formula~\eqref{Ramanujan} in Euler products asymptotics, it has not been recognized well.

There have been some conjectural and allied topics beyond the scope of the Grand Riemann Hypothesis, GRH. The asymptotic formula~\eqref{Ramanujan} is related to the Deep Riemann Hypothesis (DRH)---a refined version of the GRH---which is an outstanding open problem that may be involved with development on the GRH. The DRH states, in short, conditional convergence or divergence behavior of partial Euler products in the critical strip.

More precisely, the DRH asserts the following for the case of Dirichlet $L$-functions. If $\chi \ne \1$ and a complex number $s$ satisfies $\Re s = 1/2$ with $m$ the order of vanishing of $L(s, \chi)$ at $s$, we have
\begin{equation}\label{DirichletL}
\lim_{x \to \infty} \left((\log x)^{m} \prod_{p \leqslant x} (1-\chi(p)p^{-s})^{-1} \right)
 = \frac{L^{(m)}(s,\chi)}{e^{m\gamma}m!} \times
\begin{cases}
	\sqrt{2} & \text{if $s = 1/2 \, \wedge \, \chi^{2} = \1$}\\
	1 & \text{otherwise}
\end{cases}.
\end{equation}
This implies that the Euler product converges on the critical line $\Re s = 1/2$ except at zeros of $L(s, \chi)$. The theorem of Conrad~\cite{Conrad2005} tells us that the following two propositions are equivalent: the limit in the left hand side of~\eqref{DirichletL} exists for \textit{some} $s$ on $\Re s = 1/2$ and it exists for \textit{every} $s$ on $\Re s = 1/2$. Moreover, the conjecture~\eqref{DirichletL} is equivalent to
\begin{equation*}
\vartheta(x, \chi) \coloneqq \sum_{p \leqslant x} \chi(p) \log p = o(\sqrt{x} \log x)
\end{equation*}
as is stated by Conrad~\cite[Theorem 6.2]{Conrad2005} (the GRH says $\vartheta(x, \chi) \ll \sqrt{x}(\log x)^{2}$). Figures illustrating~\eqref{DirichletL} can be found in Section 3 of~\cite{KimuraKoyamaKurokawa2014}.

We point out, as an example, the initial version of the Birch and Swinnerton-Dyer conjecture ((A) on p.79 in the original paper~\cite{BirchSwinnertonDyer1965}) concerning Euler products asymptotics for an elliptic curve $L$-function at the central point $s = 1$ (see also~\cite{Birch1965,BirchSwinnertonDyer1963,Wiles2006}). To illustrate this conjecture, let $E_{/\Q}$ be an elliptic curve given by the equation $y^{2} = x^{3}+ax+b$ with $N$ the conductor. Also we define $a_{p} = p+1-\# E(\F_{p})$ for $p \nmid N$ and $a_{p} = p-\# E(\F_{p})$ for $p \mid N$ with $E(\F_{p})$ the set of $\F_{p}$-rational points on a minimal Weierstrass model for $E$ at $p$.
We then introduce the $L$-function
\begin{equation}\label{ellipticcurveL}
L(s, E) \coloneqq \prod_{p \mid N} \frac{1}{1-a_{p} p^{-s}} \prod_{p \nmid N} \frac{1}{1-a_{p}p^{-s}+p^{1-2s}} 
\quad \text{for} \quad \Re s > 3/2.
\end{equation}
%By appealing to the modularity theorem~\cite{BreuilConradDiamondTaylor2001,ConradDiamondTaylor1999,Diamond1996,Wiles1995}, this $L$-function extends analytically to the whole complex plane $\C$ and has a functional equation relating $L(s, E)$ to $L(2-s, E)$.
Birch and Swinnerton-Dyer conjectured that $\# E(\F_{p})$ obeys an asymptotic law
\begin{equation}\label{BirchSwinnertonDyer}
\prod_{p \leqslant x, \, p \nmid N} \frac{\# E(\F_{p})}{p} \sim A(\log x)^{r} \qquad \text{as} \qquad x \to \infty
\end{equation}
for some constant $A > 0$ (depending on $E$), where $r$ is the rank of $E$.

For notational convenience we define
\begin{equation*}
\mathrm{Prod}(x, E) = \prod_{\substack{p \mid N \\ p \leqslant x}} \frac{1}{1-a_{p}/p} 
\prod_{\substack{p \nmid N \\ p \leqslant x}} \frac{1}{1-a_{p}/p+1/p}
 = \prod_{p \nmid N, \, p \leqslant x} \frac{1}{\# E(\F_{p})/p}.
\end{equation*}
Goldfeld~\cite{Goldfeld1982} later proved that
\begin{theorem}[Goldfeld~\cite{Goldfeld1982}]\label{Goldfeld}
Let $E_{/\Q}$ be an elliptic curve. If $\mathrm{Prod}(x, E) \sim C/(\log x)^{r}$ as $x \to \infty$, where $C > 0$ and $r \in \R$, then $L(E, s)$ satisfies the GRH, i.e. $L(E, s) \ne 0$ for $\Re s > 1$, $r = \ord_{s = 1} \, L(E, s)$ and $C = L^{(r)}(1, E)/(\sqrt{2} e^{r \gamma} r!)$.
\end{theorem}
As one of byproducts of Theorem~\ref{Goldfeld}, it turned out that the conjecture~\eqref{BirchSwinnertonDyer} alludes the \textit{well-known form of the Birch and Swinnerton-Dyer conjecture}: $\ord_{s = 1} \, L(s, E) = \rank \, E(\Q)$, with $E(\Q)$ the Mordell-Weil group of a curve $E_{/\Q}$. 
%Hence~\eqref{BirchSwinnertonDyer} can be reputed as the DRH, the stronger version of the conjecture $\ord_{s = 1} \, L(s, E) = \rank \, E(\Q)$.
Notice that the extra $\sqrt{2}$ factor in the denominator of $C$ is of intrinsic importance and likewise, it appears in Ramanujan's formula (359) of~\cite{Ramanujan1997}\footnote{In general, for Hasse-Weil zeta functions, $\sqrt{2}$ is to be replaced by $\sqrt{2}^{c}$ with $c$ being the order of the pole of the second moment Euler product $L^{(2)}(s) = L(s, \mathrm{Sym}^{2})/L(s, \wedge^{2})$ at $s = 1$.}.

K. Conrad~\cite{Conrad2005} in 2005 analyzed partial Euler products for a typical $L$-function along its critical line, and demystified the $\sqrt{2}$ factor in the general context. To be precise, the appearance of this factor is governed by ``second moments.'' He found the equivalence between the Euler product asymptotics and the tight estimate $\psi_{L}(x) \coloneqq \sum_{N \frak{p}^{k} \leqslant x} (\alpha_{\frak{p}, 1}^{k}+\dots+\alpha_{\frak{p}, d}^{k}) \log N \frak{p} = o(\sqrt{x} \log x)$ which is stronger than the GRH. Concurrently Kuo and Murty~\cite{KuoMurty2005} established the equivalence between the Birch and Swinnerton-Dyer conjecture and a certain growth condition for the sum over the Frobenius eigenvalues at $p$.
%\begin{align*}
%\tilde{c}_{n} = 
%	\begin{cases}
%	\frac{\alpha_{p}^{k}+\beta_{p}^{k}}{k} & \text{$n = p^{k}$, $p$ is a prime, $k \in \Z \geqslant 1$, $p \nmid N$}\\
%	0 & \text{otherwise}
%	\end{cases}
%\end{align*}
Analogously, Akatsuka~\cite{Akatsuka2017} has recently studied the DRH\footnote{Unfortunately, as of September in 2018, we should note that a serious misunderstanding lies in the review of Akatsuka's paper in the MathSciNet, where the reviewer writes that Akatsuka studies one of equivalence conditions to the RH.} for $\zeta(s)$, capturing the divergent behavior of the partial Euler product by means of a modified logarithmic integral. That is equivalent to $\vartheta(x) = x+o(\sqrt{x} \log x)$.

Rigorous proofs of the DRH are given in the case of zeta functions over function fields (\cite{KimuraKoyamaKurokawa2014,KoyamaSuzuki2014}). The case of Selberg zeta functions (characteristic zero) is yet to be resolved and therefore it is worthwhile to examine the behavior of partial Euler products of Selberg zeta functions attached with a unitary representations. In this paper we address this problem comprehensively and describe what is indispensable to specify the behavior in several facets. In fact, despite the analogue of the RH for Selberg zeta functions, our present knowledge of these in the critical strip is not satisfactory. In the case of the trivial representation, we draw upon Ramanujan's extraordinary method, but in other cases, we have to seek other strategies and reconfigure the process of the proof.

Let $\Gamma \subset PSL(2, \R)$ be a cofinite Fuchsian group acting discontinuously on the upper half-plane $\mathbb{H}$. There is a lucid analogy between prime numbers and primitive (prime) closed geodesics on $\Gamma \backslash \mathbb{H}$. Any matrix in a hyperbolic conjugacy class in $\Gamma$ is conjugate over $SL(2, \R)$ to a matrix of the form $\begin{pmatrix} \lambda && 0 \\ 0 && \lambda^{-1} \end{pmatrix}$ with $\lambda > 1$. We then denote the norm of such a hyperbolic conjugacy class to be $N(p) = \lambda^{2}$. The Selberg zeta function of $\Gamma$ attached with its finite-dimensional unitary representation $\rho$ is defined as the following Euler product:
\begin{equation*}
Z_{\Gamma}(s, \rho) = \prod_{p} \prod_{n = 0}^{\infty} \det \left(I_{\dim \rho}-\rho(p) N(p)^{-s-n} \right),
\end{equation*}
the outer product being over the set of all primitive hyperbolic conjugacy classes $p$ in $\Gamma$. It is absolutely convergent in $\Re s > 1$, and has a meromorphic continuation to $\C$. Throughout this paper we assume $\rho$ is irreducible. The proofs for general cases are straightforward, since a decomposition of a unitary representation corresponds to that of the Selberg zeta function (see~\cite{Venkov1982,Venkov1990}). As usual, when $\rho$ is the trivial representation $\1$, we put $Z_{\Gamma}(s, \1) = Z_{\Gamma}(s)$. Such an abbreviation also applies to other functions below.

Denote
\begin{equation*}
\Theta_{\Gamma}(x, \rho) = \sum_{N(p) \leqslant x} \tr(\rho(p)) \log N(p)
\end{equation*}
and
\begin{equation*}
\Psi_{\Gamma}(x, \rho) = \sum_{N(p)^{k} \leqslant x} \tr(\rho(p)^{k}) \log N(p),
\end{equation*}
where $\sum_{N(p)^{k} \leqslant x}$ stands for the sum over all pairs of $(p, k)$ with $p$ as above and positive integers $k$ such that $N(p)^{k} \leqslant x$. Hereafter we apply this notation for the sums below unless otherwise stated. Then the \textit{prime geodesic theorem}
%, which resembles the familiar prime number theorem,
is expressed in the following form:
\begin{equation}\label{pgt}
\Psi_{\Gamma}(x, \rho) = \sum_{\frac{1}{2} < s_{j} \leqslant 1} \frac{x^{s_{j}}}{s_{j}}+\mathcal{E}_{\Gamma}(x, \rho),
%\qquad (x \to \infty)
\end{equation}
where the sum runs over the exceptional zeros of the Selberg zeta function $Z_{\Gamma}(s, \rho)$ counted with multiplicity. The error term $\mathcal{E}_{\Gamma}(x, \rho)$ plays a crucial role in our analyzing partial Euler products.
%Traditionally we set $s_{0} \geqslant s_{1} \geqslant \dotsb$.
%All non-exceptional zeros contribute to $\mathcal{E}_{\Gamma}(x, \rho)$. 
By detailed analysis of the Selberg trace formula, the error term is known to be estimated as $\mathcal{E}_{\Gamma}(x, \rho) \ll x^{\frac{3}{4}}$ for a general cofinite $\Gamma$ (cf.~\cite{Huber1961,Huber1961-2,Randol1977,Sarnak1980}~\cite[pp.426--474]{Selberg1989}). This is called the ``trivial bound.'' %It is improved to $\mathcal{E}_{\Gamma}(x) \ll x^{\frac{3}{4}}(\log \log x)^{\frac{1}{4}+\varepsilon} \ (x \not \in E)$ for some subset $E \subset \R$ having at most finite logarithmic measure~\cite[Theorem 1]{Koyama2016}. According to Hejhal~\cite[Theorem 15.13, p.252]{Hejhal1976}, the error term is expressed as
%\begin{equation*}
%O \left(x^{\frac{1+2 \alpha}{2+2 \alpha}}(\log x)^{\frac{1}{1+\alpha}} \right) \qquad (x \to \infty),
%\end{equation*}
%if $S(t) = \frac{1}{\pi} \arg Z_{\Gamma} \left(\frac{1}{2}+it \right)$ satisfies $S(t) = O \left(t^{\alpha} \right)$ for some $0 < \alpha < 1$.
%Taking $\alpha \to 1$ would recover the trivial bound. But we do not have any evidence for that assumption.
The subject concerning a formula with a better error term attracts many mathematicians and has been intensively investigated. Given that the RH holds for $Z_{\Gamma}(s, \rho)$ apart from a finite number of the exceptional zeros on the real axis, one may expect $\mathcal{E}_{\Gamma}(x, \rho) \ll x^{\frac{1}{2}+\varepsilon}$. But the abundance of the eigenvalues puts this bound out of reach\footnote{We would like to point out that over function fields, one can prove the prime geodesic theorem with the error term $O(x^{\frac{1}{2}}/\log x)$ (\cite{Nagoshi2001-2}). In~\cite{KoyamaSuzuki2014} it was confirmed that the DRH holds for Selberg zeta functions of a principal congruence group $\Gamma(A) = \{\gamma \in PGL(2, \F_{q}[t]): \gamma \equiv I \pmod{A} \}$ for $A \in \F_{q}[t]$. It follows straightforwardly that the DRH holds for other $L$-functions which are regular at $s = 1$ over global function fields (cf.~\cite{KimuraKoyamaKurokawa2014}). Accordingly Euler product asymptotics can sometimes be proved unconditionally.}.

There are several intriguing ways to conquer this problem for certain arithmetic groups. The first is Iwaniec's method which is mainly based on the Kuznetsov formula~\cite{Kuznetsov1981} and the famous result of Burgess~\cite{Burgess1963} on character sums. Iwaniec~\cite[p.187]{Iwaniec1984-2}~\cite{Iwaniec1984} first broke the 3/4-barrier for the case that $\rho = \1$ and $\Gamma = SL(2, \Z)$: $\mathcal{E}_{\Gamma}(x) \ll x^{\frac{35}{48}+\varepsilon}$, and remarked that the exponent 2/3 would follow from the Generalized Lindel\"{o}f Hypothesis for Dirichlet $L$-functions. Note that even if we employ the Kuznetsov formula and Weil's bound for the Kloosterman sums, we are blocked by the 3/4-barrier. Furthermore, it is stated that the exponent 1/2 can be deduced conditionally on a kind of extended Linnik and Selberg's conjecture (cf.~\cite{DeshouillersIwaniec1982,GoldfeldSarnak1983,Iwaniec1984,SarnakTsimerman2009,Steiner2017}). Iwaniec's proof was enhanced by Luo and Sarnak~\cite{LuoRudnickSarnak1995,LuoSarnak1995}, by appealing to the work of Hoffstein and Lockhart~\cite{HoffsteinLockhart1994}, who accordingly obtained the considerable advancement $\mathcal{E}_{\Gamma}(x) \ll x^{\frac{7}{10}+\varepsilon}$ for any congruence subgroup (see also~\cite{Cai2002}). In this connection, approaches to the spectral exponential sum $\sum_{t_{j} \leqslant T} x^{it_{j}}$ with $s_{j} = 1/2+it_{j}$ have been studied in~\cite{BalkanovaFrolenkov2018,PetridisRisager2017}.

The second is the method formulated by Soundararajan and Young~\cite{SoundararajanYoung2013}, which leaves the theory of automorphic functions on a side. They have shown the prime geodesic theorem with the exponent 25/36, which is the best known, and recover Iwaniec's ingenious result from the entirely different aspect (\cite[Theorem 3.2]{SoundararajanYoung2013}). Recently the exponent less than 25/36 has been derived in square mean sense (see~\cite{BalogBiroHarcosMaga2018,CherubiniGuerreiro2018}).

We are ending the introduction by outlining our results. For convenience, we now introduce auxiliary zeta functions:
\begin{align}\label{partial1}
Z_{\Gamma, x}(s, \rho)& = \prod_{N(p) \leqslant x} \prod_{n = 0}^{\infty} \det(I_{\dim \rho}-\rho(p) N(p)^{-s-n}), \nonumber\\
\zeta_{\Gamma, x}(s, \rho)& = \prod_{N(p) \leqslant x} \det(I_{\dim \rho}-\rho(p) N(p)^{-s})^{-1}.
\end{align}
In what follows let $\lim_{x \to \infty} \zeta_{\Gamma, x}(s, \rho) = \zeta_{\Gamma}(s, \rho) = Z_{\Gamma}(s+1, \rho)/Z_{\Gamma}(s, \rho)$ for $\Re s > 1$, and one then finds that it bears a pole at $s = s_{0} = 1$ iff $\rho = \1$.

In Section~\ref{Iwaniec}, we solve Iwaniec's problem about the distribution of ``pseudo-primes'' $N(p)$ for any congruence subgroup of $SL(2, \Z)$ within certain short intervals. Thus we arrive at a slight improvement of the error term in the explicit formula for $\Psi_{\Gamma}(x)$. These results are useful in the following sections.

The goal of Section~\ref{divEulerprod} is to explain how the partial Euler products $\zeta_{\Gamma, x}(s)$ behave in the critical strip. We execute the evaluation of the divergent behavior of $\zeta_{\Gamma, x}(s)$ as $x \to \infty$ for prototypical arithmetic subgroups of $SL(2, \R)$. That is to say, we deal with congruence subgroups of $SL(2, \Z)$, or cocompact subgroups $\Gamma_{D}$ arising from quaternion algebras constructed as follows: let $D = \left(\frac{a, b}{\Q} \right)$ be an indefinite division quaternion algebra over $\Q$. It is linearly generated by $1, \, \omega, \, \Omega, \, \omega \Omega$ over $\Q$ with $\omega^{2} = a, \,\Omega^{2} = b$ and $\omega \Omega+\Omega \omega = 0$, where $a, \, b \in \Z$ are square free and we assume that $(a, b) = 1$. %and the prime 2 is unramified.
The trace and norm maps are defined by $N(\alpha) = \alpha \overline{\alpha}$ and $\tr(\alpha) = \alpha+\overline{\alpha}$, where if $\alpha = x_{0}+x_{1} \omega+x_{2} \Omega+x_{3} \omega \Omega$, we put $\overline{\alpha} = x_{0}-x_{1} \omega-x_{2} \Omega-x_{3} \omega \Omega$. Denote by $R$ a maximal order in $D$ (or more generally an Eichler order~\cite{Eichler1955}), and for $m \geqslant 1$ let $R(m) = \{x \in R: N(x) = m \}$. In particular $R(1)$ consists of elements of norm 1 which acts on $R(m)$ by multiplication on the left and it is known that the orbits $R(1) \backslash R(m)$ are finite (see~\cite{Eichler1955}). Fix an embedding $\varphi \colon D \to M_{2}(\Q(\sqrt{a}))$,
\begin{equation*}
\varphi(\alpha) = \begin{pmatrix} \xi & \eta \\ b \overline{\eta} & \overline{\xi} \end{pmatrix} \in M_{2}(\Q(\sqrt{a}))
\end{equation*}
where $\alpha = x_{1}+x_{1} \omega+(x_{2}+x_{3} \omega) \Omega = \xi+\eta \Omega$.
In this way we define $\Gamma_{D} = \varphi(R(1)) \subseteq SL(2, \R)$. Assuming that $D$ is a division algebra guarantees that the quotient $\Gamma_{D} \backslash \mathbb{H}$ is furnished with compactness (see~\cite{GelfandGraevPyatetskiiShapiro1969}).

Koyama~\cite{Koyama1998} proved that the set of Laplace eigenvalues for $\Gamma_{D} \backslash \mathbb{H}$ is equal to the set of those of a new form (\cite{AtkinLehner1970}) for $\Gamma_{0}(N) \backslash \mathbb{H}$ with $N$ denoting the product of all ramified primes. His proof uses the Jacquet-Langlands correspondence whose image is explicitly written by Hejhal~\cite{Hejhal1985}. 
The prime geodesic theorem with exponent 7/10 holds for congruence subgroups of $SL(2, \Z)$ and $\Gamma_{D}$ (\cite{Koyama1998,LuoRudnickSarnak1995,LuoSarnak1995}). This fact is effectively used as a key ingredient in Section~\ref{divEulerprod}\footnote{The prime geodesic theorem with exponent 7/10 would hold even for $\Gamma_{D}$ in a special situation where the order of $D$ is non-maximal (see~\cite{Hejhal1985} where it is represented as $\mathcal{O}$).}.

Throughout this paper we define that
\begin{equation}\label{Gamma}
\text{$\Gamma$ is a typical arithmetic group} \quad \underset{\mathrm{def}}{\Longleftrightarrow} \quad
\text{$\Gamma$ is a congruence subgroup of $SL(2, \Z)$ or $\Gamma = \Gamma_{D}$}.
\end{equation}

There are several methods to analyze partial Euler products, we opt to tailor what Ramanujan conjured to the present problem. His dedicated argument differs from others which have been hitherto devised, and it is interesting to adopt this method to our analysis. By virtue of this we deduce one of the main results:
\begin{theorem}[Theorem~\ref{divergent}. Case I\hspace{-.1mm}I\hspace{-.1mm}I]\label{EulerProdDivergence}
Let $\Gamma$ be a typical arithmetic group defined by~\eqref{Gamma}. For $s = \sigma+it$ with $\sigma > 1/2$ we then have
\begin{equation*}
\zeta_{\Gamma, x}(s)
 = \varepsilon_{\Gamma}(s) \, \zeta_{\Gamma}(s) \exp \biggg(\Li(\Theta_{\Gamma}(x)^{1-s})
 + \sum_{\frac{1}{2} < s_{j} < 1} \Li(x^{s_{j}-s})
 - \frac{x^{-s}}{\log x} \sum_{\frac{1}{2} < s_{j} < 1} \frac{x^{s_{j}}}{s_{j}}+O(x^{\frac{1}{2}-\sigma}(\log x)^{2}) \biggg),
\end{equation*}
where $\varepsilon_{\Gamma}(s) = \pm 1$ (for details, see Remark~\ref{epsilonremark}) and $\Li(x)$ is the principal value of $\displaystyle{\int_{0}^{x} \mathrm{d} t/\log t}$.
%\begin{equation}\label{epsilon}
%\varepsilon_{\Gamma}(s) = 
%\begin{cases}
%	-1 & \text{$\zeta_{\Gamma}(\sigma)$ is real and nonpositive}\\
%	1 & \text{otherwise}
%	\end{cases}.
%\end{equation}
\end{theorem}
%Notice that the Selberg zeta function $\zeta_{\Gamma}(s)$ is real on the real axis.
%\textit{The factor $\varepsilon_{\Gamma}(s)$ is introduced to set the sign of $\zeta_{\Gamma}(\sigma)$ to justify the theorem, however the clear behavior and grounds of this factor require further investigations. See Remark 1 for more information.}
We also obtain the results for the cases of $0 < \Re s < 1/2$ and $\Re s = 1/2$ (respectively Cases I and I\hspace{-.1mm}I) that we omit here. At $s = 1/2$, which falls into Case I\hspace{-.1mm}I, the $\sqrt{2}$ factor actually emerges but it is absorbed by the error term $O((\log x)^{2})$. Although $\sqrt{2}$ ordinarily appears from some series rearrangement, the striking point in Ramanujan's method is that it unexpectedly arises from a certain limit formula.

It remains to demonstrate the behavior in the case that $\rho \ne \1$. This topic is discussed in Section~\ref{convEulerprod}. An alternative way of observing the behavior of $\zeta_{\Gamma, x}(s, \rho)$ rests on a direct expression of $\zeta_{\Gamma, x}(s, \rho)$ by a psi function with certain weighting (after picking up the coefficients of the psi function by means of Perron's integral).

\begin{theorem}\label{EulerProdConvergence}
Let $\Gamma$ be a typical arithmetic group defined by~\eqref{Gamma} and $\rho$ any unitary representation (not necessarily nontrivial). For $\Re s > 1/2$ we have
\begin{equation*}
\zeta_{\Gamma, x}(s, \rho)
 = \varepsilon_{\Gamma}(s, \rho) \, \zeta_{\Gamma}(s, \rho)
\times \exp \biggg(\sum_{\frac{1}{2} < s_{j} \leqslant 1} \Li(x^{s_{j}-s})
 + \frac{x^{\frac{1}{2}-s}}{\log x} \sum_{|t_{j}| \leqslant T}\frac{x^{it_{j}}}{s_{j}}+O \left(\frac{x^{1-\sigma}}{T} \log x \right) \biggg)
\end{equation*}
provided $1 \leqslant T \leqslant \sqrt{x}/\log x$, where $\varepsilon_{\Gamma}(s, \rho) = \pm 1$ (see Lemma~\ref{Conrad}).
%\begin{equation}\label{epsilon2}
%\varepsilon_{\Gamma}(s, \rho) = 
%\begin{cases}
%	-1 & \text{$\zeta_{\Gamma}(\sigma, \rho)$ is real and nonpositive}\\
%	1 & \text{otherwise}
%	\end{cases}.
%\end{equation}
\end{theorem}
Needless to say $\zeta_{\Gamma}(s, \rho)$ with $\rho$ nontrivial converges conditionally even if it converges, so we must treat it carefully. Notice that in view of the explicit formula for $\Psi_{\Gamma}(x)$, the term $\Li(\Theta_{\Gamma}(x)^{1-s})$ in Theorem~\ref{EulerProdDivergence} involves the sum over the eigenvalues $t_{j}$ with $\Re s_{j} = 1/2$. As things turns out later, the \textit{ultimate form} of Theorem~\ref{EulerProdDivergence} accords with that of Theorem~\ref{EulerProdConvergence}.

%\begin{theorem}
%Let $\rho$ be any nontrivial unitary representation. For $\Re s > \frac{3}{4}$, we have
%\begin{equation*}
%\lim \limits_{x \to \infty} \zeta_{\Gamma, x}(s, \rho) \prod \limits_{\frac{1}{2} < s_{j} \leqslant 1} \exp(- \Li(x^{s_{j}-s}))
% = \varepsilon_{\Gamma}(s, \rho) \, \zeta_{\Gamma}(s, \rho).
%\end{equation*}
%More precisely,
%\begin{equation*}
%\zeta_{\Gamma, x}(s, \rho)
% = \varepsilon_{\Gamma}(s, \rho) \, \zeta_{\Gamma}(s, \rho) 
%\exp \biggg(\sum_{\frac{1}{2} < s_{j} \leqslant 1} \Li(x^{s_{j}-s})+O(x^{\frac{3}{4}-\sigma} \log x) \biggg).
%\end{equation*}
%\end{theorem}

In conclusion, calculating the behavior of $\zeta_{\Gamma, x}(s, \rho)$ is relevant to the prime geodesic theorem and investigation of the location of the law-energy spectrum (still shrouded in mystery) from a statistical point of view.

\section{Brun-Titchmarsh type inequalities}\label{Iwaniec}
We are going to develop inequalities in short intervals related to counting prime geodesics, which will provide natural orientation toward deduction of the behavior of the partial Euler products $\zeta_{\Gamma, x}(s, \rho)$ for typical arithmetic groups. Those inequalities were brought up in Iwaniec's seminal work~\cite{Iwaniec1984} and he reached the following consequence (see also~\cite{ArakawaKoyamaNakasuji2002} and~\cite{Hashimoto2007}): if $\sqrt{x}(\log x)^{2} \leqslant y \leqslant x$, we have
\begin{equation}\label{Brun-Titchmarsh}
\pi_{\Gamma}(x+y)-\pi_{\Gamma}(x) \ll y,
\end{equation}
where $\Gamma = SL(2, \Z)$ and $\pi_{\Gamma}(x)$ is the pseudo-prime counting function, namely
\begin{equation*}
\pi_{\Gamma}(x) = \# \{\text{$p$ a primitive hyperbolic conjugacy class in $\Gamma$}: N(p) \leqslant x \}.
\end{equation*}
By partial summation one easily passes between the asymptotic results for $\pi_{\Gamma}(x)$, $\Theta_{\Gamma}(x)$ and $\Psi_{\Gamma}(x)$. Traditionally such estimates in short intervals are called Brun-Titchmarsh type inequalities. For $\Gamma = SL(2, \Z)$, Iwaniec~\cite{Iwaniec1984} conjectured that on the right hand side of~\eqref{Brun-Titchmarsh} $y$ can be replaced by $y/\log x$ if $\sqrt{x}(\log x)^{2} \leqslant y \leqslant x$.
Later on, the conjecture was proved by Bykovskii~\cite{Bykovskii1997} in the range $\sqrt{x} \exp(c\sqrt{\log x \log \log x}) \leqslant y \leqslant x$ with some constant $c > 0$. Recently Soundararajan and Young~\cite[Theorem 1.2]{SoundararajanYoung2013} made the following statement: if we assume the GRH for quadratic Dirichlet $L$-functions we have
\begin{equation}\label{GRH}
\pi_{\Gamma}(x+y)-\pi_{\Gamma}(x) \sim \frac{y}{\log x}, \qquad \text{where} \qquad \Gamma = SL(2, \Z)
\end{equation}
provided $\sqrt{x}(\log x)^{2+\varepsilon} \leqslant y \leqslant x$. The explicit formula for $\Psi_{\Gamma}(x)$ is predominantly proved by means of Brun-Titchmarsh type inequalities (see Lemma~\ref{slight}), and moreover ways to cope with the behavior of $\zeta_{\Gamma, x}(s, \rho)$ would go through establishing such estimates.

Incidentally the above formula is reminiscent of the case of the rational primes. The corresponding estimate is proved by using either Brun's or Selberg's sieve (see, for example,~\cite{FriedlanderIwaniec2010} and~\cite{Motohashi1983}). To be precise, Huxley~\cite{Huxley1972} has proved that the number of primes up to $x$, normally written as $\pi(x)$, admits the asymptotic formula
\begin{equation*}
\pi(x+y)-\pi(x) = \sum_{x \leqslant p \leqslant x+y} 1 = \frac{y}{\log x}+O_{\varepsilon} \left(\frac{y}{(\log x)^{2}} \right)
\end{equation*}
for $x^{\frac{1}{2}+\frac{1}{12}+\varepsilon} \leqslant y \leqslant x$. Clearly this range for $y$ falls short of being optimal.

Before stating the result we focus on a fascinating number-theoretic interpretation of numbers $N(p)$ for $SL(2, \Z)$ in terms of indefinite primitive integral quadratic forms $q(x, y) = ax^{2}+bxy+cy^{2}$ (``primitive'' means that $(a, b, c) = 1$) whose discriminants satisfy $d = b^{2}-4ac > 0$ (see~\cite{Sarnak1982}). The automorphs of such form are given by $\pm P(t, u)$ where
\begin{equation*}
P(t, u) = \begin{pmatrix} \dfrac{t-bu}{2} && -cu \\ au && \dfrac{t+bu}{2} \end{pmatrix}
\end{equation*}
with $(t, u)$ being a solution of Pell's equation $t^{2}-du^{2} = 4$. The group of automorphs of $q(x, y)$ is infinite and cyclic with a generator $P(t_{d}, u_{d})$, where $t_{d}, \, u_{d} > 0$ denotes the fundamental solution of the equation $t^{2}-du^{2} = 4$. For non-zero $u$, $P(t, u)$ is a hyperbolic element in $SL(2, \Z)$ with norm $(t+\sqrt{d}u)^{2}/4$ and trace $t$. So $P(t_{d}, u_{d})$ is a primitive hyperbolic matrix with norm $\varepsilon_{d}^{2}$ and trace $t_{d}$ where
\begin{equation*}
\varepsilon_{d} = \frac{t_{0}+\sqrt{d}u_{0}}{2}.
\end{equation*}
Sarnak's bijection $q(x, y) \mapsto P(t, u)$ sends primitive integral quadratic forms to hyperbolic conjugacy classes of $SL(2, \Z)$. More precisely it indicates that the norms of primitive classes are $\varepsilon_{d}^{2}$ with multiplicity $h(d)$, the class number, where $d \in \mathcal{D} = \{d > 0: d \equiv 0, 1 \pmod 4, \, \text{$d$ not a square} \}$ is the set of positive ring discriminants (see~\cite[Corollary~1.5]{Sarnak1982}). Generally for any typical arithmetic group some modification is required as well (for more information see~\cite{Hashimoto2007,Sarnak1982}).

%Denote by $\mathcal{D}$ the set of positive ring discriminants, that is $\{d > 0: d \equiv 0, 1 \pmod 4, \, \text{$d$ not a square} \}$. Let $h(d)$ be the number of inequivalent classes of such forms and $\varepsilon_{d} = \frac{1}{2}(n+\sqrt{d}m)$ the fundamental solution of Pell's equation $n^{2}-dm^{2} = 4$.
%Sarnak's bijection indicates that the norms of primitive classes are $\varepsilon_{d}^{2}$ with multiplicity $h(d)$ where $d \in \mathcal{D}$, see~\cite{Sarnak1982}.

Typical arithmetic groups are of interest in the sense that the Selberg zeta functions of $\Gamma$ are characterized by ``arithmetic expressions.''
\begin{lemma}\label{ArithmeticExpression}
For a typical arithmetic group $\Gamma$ defined by~\eqref{Gamma} we have
\begin{equation*}
\frac{Z_{\Gamma}^{\prime}}{Z_{\Gamma}}(s)
 = \sum_{d \in \mathcal{D}} \sum_{k \geqslant 1} 
\delta_{\Gamma}(d, k) h(d) \frac{2 \log \varepsilon_{d}}{1-\varepsilon_{d}^{-2k}} \varepsilon_{d}^{-2ks}
\end{equation*}
for certain positive integer coefficients $\delta_{\Gamma}(d, k)$.
\end{lemma}
\begin{proof}
For $\Gamma_{D}$ arising from a quaternion algebra, this expression is proved in~\cite{ArakawaKoyamaNakasuji2002}. So we concentrate on the case of congruence subgroups in $SL(2, \Z)$. The proof uses the Venkov-Zograf theorem~\cite{VenkovZograf1983}. It says that if $\Gamma_{1}$ is a subgroup of $\Gamma$ with a finite index in $\Gamma$ and $\rho_{1}$ is a finite-dimensional unitary representation of $\Gamma_{1}$, then the functoriality
\begin{equation}\label{functoriality}
Z_{\Gamma_{1}}(s, \rho_{1}) = Z_{\Gamma}(s, \mathrm{Ind}_{\Gamma_{1}}^{\Gamma} \, \rho_{1}).
\end{equation}
holds; specially one has
\begin{equation*}
Z_{\Gamma}(s) = Z_{SL(2, \Z)}(s, \mathrm{Ind}_{\Gamma}^{SL(2, \Z)} \, \1).
\end{equation*}
Further, we can use $\tr ((\mathrm{Ind}_{\Gamma}^{SL(2, \Z)} \, \1)(\gamma))$ for each congruence subgroup $\Gamma$. A similar discussion to Section 2.2 of~\cite{Hashimoto2007} leads to the lemma.
\end{proof}
Of course, there are expressions corresponding to $\zeta_{\Gamma}^{\prime}/\zeta_{\Gamma}(s)$.

We can show that Bykovskii's excellent achievement also holds for typical arithmetic groups.
\begin{theorem}\label{Bykovskii}
Let $\Gamma$ be a typical arithmetic group. In the range $x^{\frac{1}{2}+\varepsilon} \leqslant y \leqslant x$ we have
\begin{equation*}
\pi_{\Gamma}(x+y)-\pi_{\Gamma}(x) \ll \frac{y}{\log x} \qquad \text{and} \qquad 
\Psi_{\Gamma}(x+y)-\Psi_{\Gamma}(x) \ll y.
\end{equation*}
\end{theorem}
Bykovskii gave a remark that this theorem is available in the plausible range $\sqrt{x} \exp(c \sqrt{\log x \log \log x}) \leqslant y \leqslant x$ for some constant $c > 0$ (see Remark at the bottom of his paper).
\begin{proof}
Letting
\begin{equation*}
\Pi_{\Gamma}(x) = \sum_{N(p)^{k} \leqslant x} k^{-1},
\end{equation*}
which is allied to $\pi_{\Gamma}(x)$, Lemma~\ref{ArithmeticExpression} yields
\begin{equation*}
\Pi_{\Gamma}(x) = \sum_{\substack{d \in \mathcal{D} \\ \varepsilon(d)^{2k} \leqslant x}} \delta_{\Gamma}(d, k) k^{-1} h(d).
\end{equation*}
Clearly $\delta_{SL(2, \Z)}(d, k)$ turns out to be equal to 1 from the correspondence described above. For congruence subgroups, the factor $\delta_{\Gamma}(d, k)$ is bounded by the inequality $0 \leqslant \delta_{\Gamma}(d, k) \leqslant [SL(2, \Z), \Gamma]$ for any $d \in \mathcal{D}$ and $k \geqslant 1$ (see~\cite{Hashimoto2007}), while for $\Gamma_{D}$, it is the same as in~\cite{ArakawaKoyamaNakasuji2002}. Thus we have
\begin{equation*}
\Pi_{\Gamma}(x+y)-\Pi_{\Gamma}(x) \ll \Pi_{SL(2, \Z)}(x+y)-\Pi_{SL(2, \Z)}(x).
\end{equation*}
This can be rewritten in the form
\begin{equation*}
\pi_{\Gamma}(x+y)-\pi_{\Gamma}(x) \ll \pi_{SL(2, \Z)}(x+y)-\pi_{SL(2, \Z)}(x)
\end{equation*}
by the transformation formulas
\begin{equation*}
\Pi_{\Gamma}(x) = \sum_{k \geqslant 1} k^{-1} \pi_{\Gamma}(x^{\frac{1}{k}}) \qquad \text{and} \qquad
\pi_{\Gamma}(x) = \Pi_{\Gamma}(x)-\sum_{k \geqslant 2} k^{-1} \Pi_{\Gamma}(x^{\frac{1}{k}}).
\end{equation*}
We then exploit Bykovskii's theorem~\cite[Theorem 1]{Bykovskii1997}, proving Theorem~\ref{Bykovskii}.
\end{proof}

We remark that for $\Gamma = SL(2, \Z)$, the class number formula tells us that
\begin{equation}\label{Kuznetsov}
\Psi_{\Gamma}(x) = \sum_{2 < t \leqslant X} \sum_{t^{2}-du^{2} = 4} h(d) \log \varepsilon_{d}^{2}
 = 2 \sum_{2 < t \leqslant X} \sum_{du^{2} = t^{2}-4} \sqrt{d} L(1, \chi_{d})
\end{equation}
with $X = \sqrt{x}+1/\sqrt{x}$ and $\chi_{d}$ being the real character of discriminant $d$. Moreover the inner sum in~\eqref{Kuznetsov} may be expressed as $\sqrt{t^{2}-4} \, \mathcal{L}(1, t^{2}-4)$ for a Dirichlet series
\begin{equation*}
\mathcal{L}(s, \delta) = \frac{\zeta(2s)}{\zeta(s)} \sum_{q \geqslant 1} \rho_{q}(\delta) q^{-s}
 = \sum_{q \geqslant 1} \lambda_{q}(d) q^{-s}
\end{equation*}
which is defined for all discriminants $\delta$ (cf.~\cite{Zagier1977}), where
\begin{equation*}
\rho_{q}(\delta) = \# \{x \tpmod{2q}: x^{2} \equiv \delta \tpmod{4q} \}
\end{equation*}
and
\begin{equation*}
\lambda_{q}(d) = \sum_{q_{1}^{2}q_{2}q_{3} = q} \mu(q_{2}) \rho_{q_{3}}(d).
\end{equation*}
From~\cite[Lemma 2.1]{SoundararajanYoung2013} we reach
\begin{equation}\label{Kuznetsov2}
\Psi_{\Gamma}(x) = 2 \sum_{2 < n \leqslant X} \sqrt{n^{2}-4} \, \mathcal{L}(1, n^{2}-4).
\end{equation}
The expression~\eqref{Kuznetsov2} is decisively featured in Bykovskii's proof for the problem of Iwaniec (see~\cite[page 142]{Iwaniec1984}~\cite{Kuznetsov1978}, and~\cite[Proposition 2.2]{SoundararajanYoung2013}). It should be pointed out that Gauss noticed the behavior
\begin{equation*}
\sum_{\substack{d \in \mathcal{D} \\ d \leqslant x}} 2h(d) \log \varepsilon_{d}
 = \frac{\pi^{2} x^{\frac{3}{2}}}{9 \, \zeta(3)}+O(x \log x),
\end{equation*}
and it concerns with a mean value of $L(1, \chi_{d})$ (cf.~\cite[p.518]{Hejhal1983}~\cite{Shintani1975,Siegel1944}).

%\begin{theorem}
%Assume the GRH for quadratic Dirichlet L-functions. In the range $\sqrt{x}(\log x)^{2+\varepsilon} \leqslant y \leqslant x$ we have
%\begin{equation*}
%\pi_{\Gamma}(x+y)-\pi_{\Gamma}(x) \ll \frac{y}{\log x} \quad \text{and} \quad \Psi_{\Gamma}(x+y)-\Psi_{\Gamma}(x) \ll y
%\end{equation*}
%\end{theorem}

As mentioned before, the difference $\pi_{\Gamma}(x+y)-\pi_{\Gamma}(x)$ for $\Gamma = SL(2, \Z)$ is asymptotically equal to $y/\log x$ as $x \to \infty$. It seems that the following is interesting and challenging:
\begin{problem}
Obtain an asymptotic formula for $\pi_{\Gamma}(x+y)-\pi_{\Gamma}(x)$ in $\sqrt{x}(\log x)^{2} < y < x$ for any congruence subgroup $\Gamma$.
\end{problem}

Note that the trace $\tr(p)$ must be an integer $n > 2$, whence $N(p) = ((n+\sqrt{n^{2}-4})/2)^{2} = n^{2}-2+O(n^{-1})$. Thus unlike the rational primes which have an average gap of $\log x$, the norms $N(p)$ are widely spaced with each possible norm appearing with high multiplicity $h(d)$. Notice that the estimate $\pi_{\Gamma}(x+y)-\pi_{\Gamma}(x) \sim y/\log x$ cannot always hold true for $y \leqslant \sqrt{x}$.
The Brun-Titchmarsh type inequality for $1 \leqslant y \leqslant x$ is similarly considered by omitting the congruence condition at the expense of crudity in an analogous proof to Iwaniec's~\cite[Lemma 4]{Iwaniec1984} (see also~\cite{ArakawaKoyamaNakasuji2002}). Golubeva~\cite{Golubeva1999} proved the lower and upper bounds for $\pi_{\Gamma}((x+y)^{2})-\pi_{\Gamma}(x^{2})$ (note that the function ``$\pi_{\Gamma}(X)$'' in his paper is different from the conventional notation).
%\begin{proposition}
%Let $\Gamma$ be a typical arithmetic group. For $1 \leqslant y \leqslant x$ we then have
%\begin{equation*}\label{crude}
%\pi_{\Gamma}(x+y)-\pi_{\Gamma}(x) \ll y+\sqrt{x}(\log x)^{2}.
%\end{equation*}
%\end{proposition}

From Theorem~\ref{Bykovskii} we are led to the following explicit formula with parameter $T$ to be chosen later which we can vary within $1 \leqslant T \leqslant \sqrt{x}/\log x$ (cf.~\cite{Iwaniec1984,Koyama1998}):
\begin{lemma}\label{slight}
Let $\Gamma$ be a typical arithmetic group. We then have
\begin{equation}\label{explicitformula}
\Psi_{\Gamma}(x, \rho) = \sum_{\frac{1}{2} < s_{j} \leqslant 1} \frac{x^{s_{j}}}{s_{j}}
 + \sqrt{x} \sum_{|t_{j}| \leqslant T} \frac{x^{it_{j}}}{s_{j}}+O \left(\frac{x}{T}(\log x)^{2} \right),
\end{equation}
provided $1 \leqslant T \leqslant \sqrt{x}/\log x$, where $s_{j}$ in the first sum are the exceptional zeros of $Z_{\Gamma}(s)$ with $1/2 < s_{j} \leqslant 1$ and those in the second sum are non-exceptional zeros with $\Re s_{j} = 1/2$ and $t_{j} = \Im s_{j}$. The condition $|t_{j}| \leqslant T$ means that the sum runs through the spectral parameters $\pm t_{j}$ with $t_{j} \in \R, \, t_{j} \leqslant T$ (recall our sign convention $t_{j} > 0$ for the non-exceptional eigenvalues).
\end{lemma}
The proof exploits Iwaniec's method in~\cite{Iwaniec1984}, Theorem~\ref{Bykovskii} and standard results for $\zeta_{\Gamma}(s, \rho)$ from the Phragm\'{e}n-Lindel\"{o}f convexity principle. Generalization to the explicit formula with $\rho$ is ensured by the basic inequality $|\tr(\rho(p))| \leqslant \dim \rho$. For $\rho = \1$, Lemma~\ref{slight} is a slight improvement of~\cite[Lemma~1]{Iwaniec1984}.

\section{Behavior of Euler products: Unearthing the vision of Ramanujan}\label{divEulerprod}%%%%%%%%%%%%%%%%%
This section is devoted to an investigation of the behavior of the partial Euler product $\zeta_{\Gamma, x}(s)$ with $\Gamma$ being a typical arithmetic group. We employ the sophisticated machinery of Ramanujan to see that $\zeta_{\Gamma, x}(s)$ is more precisely approximated by means of $\Li(\Theta_{\Gamma}(x)^{1-s})$ rather than $\Li(x^{1-s})$. Some intimate descriptions on his method are written in~\cite{AndrewsBerndt2013} and~\cite{Robin1991}. Before embarking our Euler product asymptotics, we note that in his method, there is ambiguity on derivation of a so-called ``constant term.'' Note that this constant term is justified by the analogy of the theorem of Conrad for $\rho = \1$ (Lemma~\ref{Conrad}, Section~\ref{convEulerprod}). One reason for restricting $\Gamma$ is that the distribution of the point spectrum on $\Gamma \backslash \mathbb{H}$ with $\Gamma$ cofinite remains elusive. The following reveals the divergent behavior of the partial Euler products in the critical strip:
\begin{theorem}\label{divergent}
Let $\Gamma$ be a typical arithmetic group. We then have
\begin{description}
\item[Case I. $\Re s < 1/2$] 
\begin{multline*}
\log \zeta_{\Gamma, x}(s) = \Li(\Theta_{\Gamma}(x)^{1-s})+\frac{1}{2} \Li(\Theta_{\Gamma}(x)^{1-2s})+\dots
 + \frac{1}{n} \Li(\Theta_{\Gamma}(x)^{1-ns})\\
 + \sum_{\frac{1}{2} < s_{j} < 1} \sum_{1 \leqslant \nu \leqslant n} \frac{1}{\nu} \Li(x^{s_{j}-\nu s})
 - \frac{1}{\log x} \sum_{\frac{1}{2} < s_{j} < 1} \sum_{1 \leqslant \nu \leqslant n} \frac{x^{s_{j}-\nu s}}{\nu s_{j}}
 + O(x^{\frac{1}{2}-\sigma}(\log x)^{2})
\end{multline*}
with $n = \left[1+1/(2|\sigma|) \right]$;\\
\item[Case I\hspace{-0.1mm}I. $\Re s = 1/2$] 
\begin{equation*}
\zeta_{\Gamma, x}(s) = \exp \biggg(\Li(\Theta_{\Gamma}(x)^{1-s})+\sum_{\frac{1}{2} < s_{j} < 1} \Li(x^{s_{j}-s})
 - \frac{x^{-s}}{\log x} \sum_{\frac{1}{2} < s_{j} < 1} \frac{x^{s_{j}}}{s_{j}}+O((\log x)^{2}) \biggg);
\end{equation*}
\item[Case I\hspace{-0.1mm}I\hspace{-0.1mm}I. $\Re s > 1/2$] 
\begin{equation*}
\zeta_{\Gamma, x}(s)
 = \varepsilon_{\Gamma}(s) \, \zeta_{\Gamma}(s) 
\exp \biggg(\Li(\Theta_{\Gamma}(x)^{1-s})+\sum_{\frac{1}{2} < s_{j} < 1} \Li(x^{s_{j}-s})
 - \frac{x^{-s}}{\log x} \sum_{\frac{1}{2} < s_{j} < 1} \frac{x^{s_{j}}}{s_{j}}+O(x^{\frac{1}{2}-\sigma}(\log x)^{2}) \biggg).
\end{equation*}
\end{description}
\end{theorem}

\begin{proof}
We begin with considering the partial summation that if $\Phi^{\prime}(x)$ is continuous between $N(p_{1})$ and $x$, then
\begin{multline*}
\Phi(N(p_{1})) \log N(p_{1})+\Phi(N(p_{2})) \log N(p_{2})+\dots+\Phi(N(p_{r})) \log N(p_{r})\\
 = \Phi(x) \Theta_{\Gamma}(x)-\int_{N(p_{1})}^{x} \Phi^{\prime}(t) \Theta_{\Gamma}(t) \mathrm{d} t,
\end{multline*}
where $N(p_{1}) \leqslant N(p_{2}) \leqslant \dotsb \leqslant N(p_{r})$ is an ascending sequence of consecutive norms, and $N(p_{r})$ is the largest norm below $x$.
%We could compute $N(p_{1})$ by the interpretation of the norms above. 
%For the modular surface, we infer that the sequence of the corresponding norms is
%\begin{equation*}
%\frac{7+3 \sqrt{5}}{2}, \ 7+4 \sqrt{3}, \ \frac{23+5 \sqrt{21}}{2}, \ 17+12 \sqrt{2}, \ 17+12 \sqrt{2}, \dotsb,
%\end{equation*}
%on taking account of an interpretation of norms in terms of primitive binary quadratic forms.
Integration by parts and Taylor's theorem give
\begin{multline*}
\Phi(x) \Theta_{\Gamma}(x)-\int_{N(p_{1})}^{x} \Phi^{\prime}(t) \Theta_{\Gamma}(t) \mathrm{d} t
 = \mathrm{const}+\int_{N(p_{1})}^{x} \Phi(t) \mathrm{d} t-(x-\Theta_{\Gamma}(x)) \Phi(x)
 + \int_{N(p_{1})}^{x} \Phi^{\prime}(t)(t-\Theta_{\Gamma}(t)) \mathrm{d} t
\end{multline*}
where ``const'' depends solely of $\Phi$ and $\Gamma$. Further, the prime geodesic theorem with exponent 7/10, which is good enough for our purpose here, leads us to
\begin{equation}\label{volcano}
\int_{N(p_{1})}^{\Theta_{\Gamma}(x)} \Phi(t) \mathrm{d} t
 = \int_{N(p_{1})}^{x} \Phi(t) \mathrm{d} t-(x-\Theta_{\Gamma}(x)) \Phi(x)
 + \frac{1}{2}(x-\Theta_{\Gamma}(x))^{2} \Phi^{\prime}(x+O(x^{\frac{7}{10}}(\log x)^{2})).
\end{equation}
Gathering together these equations, we find that
\begin{multline}
\Phi(N(p_{1})) \log N(p_{1})+\Phi(N(p_{2})) \log N(p_{2})+\dots+\Phi(N(p_{r})) \log N(p_{r})\\
 = \mathrm{const}+\int_{N(p_{1})}^{\Theta_{\Gamma}(x)} \Phi(t) \mathrm{d} t
 + \int_{N(p_{1})}^{x} \Phi^{\prime}(t)(t-\Theta_{\Gamma}(t)) \mathrm{d} t
 - \frac{1}{2}(x-\Theta_{\Gamma}(x))^{2} \Phi^{\prime}(x+O(x^{\frac{7}{10}}(\log x)^{2})).
\end{multline}
%This bound is derived by strengthening a step of Iwaniec's proof employing Hoffstein and Lockhart's bound for the first Fourier coefficient of a Maass form for $\Gamma_{0}(q)$, the Hecke congruence group (see~\cite{HoffsteinLockhart1994}). Notice that for $SL(2, \Z)$, Soundararajan and Young~\cite{SoundararajanYoung2013} confirmed the bound $O(x^{\frac{25}{36}+\varepsilon})$.

A variety of characteristic asymptotic formulas can be obtained by choosing a suitable test function $\Phi$. We now set $\Phi(x) = 1/(x^{s}-1)$ to consider the growth of the partial Euler products, whence
\begin{multline}\label{asymptotic}
\frac{\log N(p_{1})}{N(p_{1})^{s}-1}+\frac{\log N(p_{2})}{N(p_{2})^{s}-1}+\dots+\frac{\log N(p_{r})}{N(p_{r})^{s}-1}\\
 = \mathrm{const}+\int_{N(p_{1})}^{\Theta_{\Gamma}(x)} \frac{\mathrm{d} t}{t^{s}-1}
 - s \int_{N(p_{1})}^{x} \frac{t-\Theta_{\Gamma}(t)}{t^{1-s}(t^{s}-1)^{2}} \mathrm{d} t+O(x^{\frac{2}{5}-\sigma}(\log x)^{4}).
\end{multline}
In order to reduce $t-\Theta(t)$, we invoke Lemma~\ref{slight}. If we specify $T \coloneqq \sqrt{t}(\log t)^{-1}$ we get
\begin{equation*}
t-\Theta_{\Gamma}(t) = t-\Psi_{\Gamma}(t)+\Psi_{\Gamma}(\sqrt{t})+\dotsb
 = \sqrt{t}-\sum_{\frac{1}{2} < s_{j} < 1} \frac{t^{s_{j}}}{s_{j}}
 - \sum_{|t_{j}| \leqslant T} \frac{t^{s_{j}}}{s_{j}}+O(\sqrt{t}(\log t)^{3}),
\end{equation*}
where $\sqrt{t}$ can be absorbed by the error term but we leave it to see later a cancellation with another term (equation~\eqref{root}). Inasmuch as the non-exceptional zeros obey Weyl's asymptotic law $N_{\Gamma}(T) = \# \{j: |t_{j}| \leqslant T \} \ll T^{2}$, the contributions from the sums over the exceptional and non-exceptional eigenvalues are calculated as
\begin{equation*}
\sum_{\frac{1}{2} < s_{j} < 1} \frac{1}{s_{j}} \int_{N(p_{1})}^{x} \frac{t^{s_{j}} \mathrm{d} t}{t^{1-s}(t^{s}-1)^{2}}
 = \sum_{\frac{1}{2} < s_{j} < 1} \sum_{\nu \geqslant 1} \frac{\nu x^{s_{j}-\nu s}}{s_{j}(s_{j}-\nu s)}+\mathrm{const}
\end{equation*}
and
\begin{equation*}
\sum_{|t_{j}| \leqslant T} \frac{1}{s_{j}} \int_{N(p_{1})}^{x} \frac{t^{s_{j}} \mathrm{d} t}{t^{1-s}(t^{s}-1)^{2}}
 = O(x^{\frac{1}{2}-\sigma} \log x)+\mathrm{const},
\end{equation*}
respectively. A finer asymptotic expansion for $N_{\Gamma}(T)$ is available (\cite{Hejhal1983,Steil1994,Venkov1990}), and also see~\cite{Luo2001} and~\cite{PhillipsSarnak1985} for generic situation. Therefore we get
\begin{multline}\label{norms}
\frac{\log N(p_{1})}{N(p_{1})^{s}-1}+\frac{\log N(p_{2})}{N(p_{2})^{s}-1}+\dots+\frac{\log N(p_{r})}{N(p_{r})^{s}-1}\\
 = \mathrm{const}+\int_{N(p_{1})}^{\Theta_{\Gamma}(x)} \frac{\mathrm{d} t}{t^{s}-1}
 - s \int_{N(p_{1})}^{x} \frac{t^{-\frac{1}{2}+s}}{(t^{s}-1)^{2}} \mathrm{d} t
 + s \sum_{\frac{1}{2} < s_{j} < 1} \sum_{\nu \geqslant 1} \frac{\nu x^{s_{j}-\nu s}}{s_{j}(s_{j}-\nu s)}
 + O(x^{\frac{1}{2}-\sigma}(\log x)^{3}),
\end{multline}
where ``const'' depends only on $s$ and $\Gamma$. Furthermore, by a simple computation and truncating the unnecessary order we find that the sum of the two integrals and the double sum is equal to
\begin{equation}\label{Theta}
%\int^{\Theta_{\Gamma}(x)} \frac{\mathrm{d} t}{t^{s}-1}
% - s \int \frac{x^{-\frac{1}{2}+s}}{(x^{s}-1)^{2}} \mathrm{d} x
% + s \sum_{\frac{1}{2} < s_{j} \leqslant 1} \sum_{\nu \geqslant 1} \frac{\nu x^{s_{j}-\nu s}}{s_{j}(s_{j}-\nu s)}\\
\frac{\Theta_{\Gamma}(x)^{1-s}}{1-s}+\frac{\Theta_{\Gamma}(x)^{1-2s}}{1-2s}
 + \dots+\frac{\Theta_{\Gamma}(x)^{1-ns}}{1-ns}
 - \frac{2s x^{\frac{1}{2}-s}}{1-2s}+E_{\Gamma}(s, x)+\mathrm{const}+O(x^{\frac{1}{2}-\sigma}(\log x)^{3}),
\end{equation}
where $n = \left[1+1/(2|\sigma|) \right]$ with $[x]$ being the floor function and
\begin{equation*}
E_{\Gamma}(s, x) = s \sum_{\frac{1}{2} < s_{j} < 1} \sum_{1 \leqslant \nu \leqslant n} 
\frac{\nu x^{s_{j}-\nu s}}{s_{j}(s_{j}-\nu s)}.
\end{equation*}
One thing to be aware of is that the $\Theta_{\Gamma}(x)$'s except for that in the first term in~\eqref{Theta} are replaced by $x$, assuming the RH, in Ramanujan's paper~\cite{Ramanujan1997}. Whereas, in our situation, those $\Theta_{\Gamma}(x)$'s cannot be done so, owing to the possible exceptional eigenvalues. Thus we easily get
\begin{multline}\label{theta3}
\frac{\log N(p_{1})}{N(p_{1})^{s}-1}+\frac{\log N(p_{2})}{N(p_{2})^{s}-1}+\dots+\frac{\log N(p_{r})}{N(p_{r})^{s}-1}\\
 = \mathrm{const}+\frac{\Theta_{\Gamma}(x)^{1-s}}{1-s}+\frac{\Theta_{\Gamma}(x)^{1-2s}}{1-2s}+\dots
 + \frac{\Theta_{\Gamma}(x)^{1-ns}}{1-ns}
 - \frac{2s x^{\frac{1}{2}-s}}{1-2s}+E_{\Gamma}(s, x)+O(x^{\frac{1}{2}-\sigma}(\log x)^{3}).
\end{multline}
If $\Re s > 1/2$ and the logarithmic derivative of an analogous limit formula of the form~\cite[Corollary 4.5]{Akatsuka2017} is given (see Lemma~\ref{Conrad}), we have $\mathrm{const} = -\zeta_{\Gamma}^{\prime}/\zeta_{\Gamma}(s)$.

We then replace $s = \sigma+it$ by $u+it$ and integrate~\eqref{theta3} with respect to $u$ from $\infty$ to $\sigma$. Hence we get from
%\begin{equation}\label{log integral}
%\int \frac{x^{a+bs}}{a+bs} \mathrm{d} s = \frac{1}{b} \Li(x^{a+bs})
%\end{equation}
%and
\begin{equation*}
\int_{\infty}^{\sigma} E_{\Gamma}(u+it, x) \mathrm{d} u
 = -\sum_{\frac{1}{2} < s_{j} < 1} \sum_{1 \leqslant \nu \leqslant n} \frac{1}{\nu} \Li(x^{s_{j}-\nu s})
 + \frac{1}{\log x} \sum_{\frac{1}{2} < s_{j} < 1} \sum_{1 \leqslant \nu \leqslant n} \frac{x^{s_{j}-\nu s}}{\nu s_{j}}
\end{equation*}
that
\begin{multline}\label{theta2}
\log ((1-N(p_{1})^{-s})(1-N(p_{2})^{-s}) \dotsm (1-N(p_{r})^{-s}))\\
 = -\log \varepsilon_{\Gamma}(s) \, \zeta_{\Gamma}(s)
 - \Li(\Theta_{\Gamma}(x)^{1-s})-\frac{1}{2} \Li(\Theta_{\Gamma}(x)^{1-2s})-\dots
 - \frac{1}{n} \Li(\Theta_{\Gamma}(x)^{1-ns})\\+\frac{1}{2} \Li(x^{\frac{1}{2}-s})
 - \sum_{\frac{1}{2} < s_{j} < 1} \sum_{1 \leqslant \nu \leqslant n} \frac{1}{\nu} \Li(x^{s_{j}-\nu s})
 + \frac{1}{\log x} \sum_{\frac{1}{2} < s_{j} < 1} \sum_{1 \leqslant \nu \leqslant n} \frac{x^{s_{j}-\nu s}}{\nu s_{j}}
 + O(x^{\frac{1}{2}-\sigma}(\log x)^{2}),
\end{multline}
where $\varepsilon_{\Gamma}(s) = \pm 1$ (see Remark~\ref{epsilonremark}). Classifying~\eqref{theta2} into $0 < \Re s < 1/2$, $\Re s = 1/2$ and $\Re s > 1/2$ and using the explicit formula, the proof is completed.
\end{proof}

Theorem~\ref{divergent} suggests that $\zeta_{\Gamma, x}(s) \exp(-\Li(x^{1-s}))$ would not converge on the critical line $\Re s = 1/2$ even though the expected estimate $\Theta_{\Gamma}(x) = x+O(x^{\frac{1}{2}}(\log x)^{A})$ for some $A \geqslant 1$ is fulfilled. That is to say, as for the error term in Case I\hspace{-.1mm}I, it is difficult to go beyond the order of magnitude of $(\log x)^{2}$ as Ramanujan did. Because of this incompleteness the $\sqrt{2}$ factor cannot be attained. Note that this factor comes up from
\begin{equation}\label{root}
\lim_{s \to \frac{1}{2}} \left(\Li(x^{\frac{1}{2}-s})-\Li(\Theta_{\Gamma}(x)^{1-2s}) \right) = \log 2.
\end{equation}
%where
%\begin{equation*}
%\Li(\Theta_{\Gamma}(x)^{1-2s})
% = \Li(x^{1-2s})+\frac{x^{-2s}}{\log x} \sum_{\frac{1}{2} < s_{j} \leqslant 1} \frac{x^{s_{j}}}{s_{j}}
% + \frac{x^{\frac{1}{2}-2s}}{\log x} \sum_{|t_{j}| \leqslant T} \frac{x^{it_{j}}}{s_{j}}+O \left(\frac{x^{1-2s}}{T} \log x \right)
%\end{equation*}

\begin{remark}\label{epsilonremark} \mbox{}
\begin{enumerate}
\item[(A)] The reason why $\zeta_{\Gamma}(s)$ is accompanied by the factor $\varepsilon_{\Gamma}(s)$, purely emerges from the uncertainty of the sign of $\zeta_{\Gamma}(\sigma) \in \R$ (recall that $\varepsilon_{\Gamma}(s)$ appears for the case of $\sigma = \Re s > 1/2$). The factor $\varepsilon_{\Gamma}(s)$ is introduced to set the sign of $\zeta_{\Gamma}(\sigma)$ to justify the theorem. We conjecture that
\begin{equation}\label{epsilon}
\varepsilon_{\Gamma}(s) = 
\begin{cases}
	-1 & \text{$\zeta_{\Gamma}(\sigma)$ is real and nonpositive}\\
	1 & \text{otherwise}
	\end{cases}.
\end{equation}
Notice that $\zeta_{\Gamma}(s)$ is real on the real axis. As we discuss afterward in Section~\ref{convEulerprod}, the overall picture of zeros and poles of $\zeta_{\Gamma}(s)$ is clear, so that we have $\zeta(\sigma) \ne 0$ in $1/2 < \sigma < 1$. Taking account of a simple pole of $\zeta_{\Gamma}(s)$ at $s = 1$, we see that $\zeta_{\Gamma}(\sigma) < 0$ in $1/2 < \sigma < 1$ on Selberg's eigenvalue conjecture, specifying $\varepsilon_{\Gamma}(s) = -1$ in the region. 
If we apply Akatsuka's method~\cite{Akatsuka2017}, the equation~\eqref{epsilon} could be derived directly for $1/2 < \sigma < 1$.\\
\item[(B)] Indeed the minus sign in Ramanujan's formula~\eqref{Ramanujan} comes from the fact that $\zeta(s) < 0$ uniformly in $1/2 < s < 1$. For $s \in \C$ with $1/2 < \sigma < 1$, Akatsuka~\cite[Corollary 3.6]{Akatsuka2017} checked that $-\zeta(s)$ arises similarly (except for the point $s = 1/2$ where $-\sqrt{2} \, \zeta(1/2)$ appears). On the other hand, for Dirichlet $L$-functions (\cite[Corollary 5.6]{Conrad2005}), we have $L(\sigma, \chi) \geqslant 0$ for real characters $\chi$ in the same region, so that the minus sign does not appear anywhere (see also~\eqref{DirichletL}). Generally the difference as to whether the minus sign comes out or not depends on whether zeta or $L$-functions have a simple pole at $s = 1$ or not. Anyway, grounds and roles of $\varepsilon_{\Gamma}(s)$ require further investigations.
\end{enumerate}
\end{remark}

One now expands $\Theta_{\Gamma}(x)$'s so that we will find a concise version of Theorem~\ref{divergent}. For clarity we take Case I\hspace{-0.1mm}I\hspace{-0.1mm}I. Everything goes through as well for the other cases.
\begin{corollary}
Let $\Gamma$ be a typical arithmetic group. For $\Re s > 1/2$ we then have
\begin{equation}
\zeta_{\Gamma, x}(s)
 = \varepsilon_{\Gamma}(s) \, \zeta_{\Gamma}(s) \exp \left(\Li(x^{1-s})+O(x^{\frac{7}{10}-\sigma} \log x) \right).
\end{equation}
\end{corollary}

\begin{proof}
Again by appealing to the explicit formula~\eqref{explicitformula}, there follows
\begin{equation*}
\Li(\Theta_{\Gamma}(x)^{1-s})
 = \Li(x^{1-s})+\frac{x^{-s}}{\log x} \sum_{\frac{1}{2} < s_{j} \leqslant 1} \frac{x^{s_{j}}}{s_{j}}
 + \frac{x^{\frac{1}{2}-s}}{\log x} \sum_{|t_{j}| \leqslant T} \frac{x^{it_{j}}}{s_{j}}
 + O \left(\frac{x^{1-\sigma}}{T} \log x \right)
\end{equation*}
within $1 \leqslant T \leqslant \sqrt{x}/\log x$.
Consequently, we deduce
\begin{equation}\label{finalasymptotics}
\zeta_{\Gamma, x}(s)
 = \varepsilon_{\Gamma}(s) \, \zeta_{\Gamma}(s) \exp \biggg(\sum_{\frac{1}{2} < s_{j} \leqslant 1} \Li(x^{s_{j}-s})
 + \frac{x^{\frac{1}{2}-s}}{\log x} \sum_{|t_{j}| \leqslant T} \frac{x^{it_{j}}}{s_{j}}
 + O \left(\frac{x^{1-\sigma}}{T} \log x \right) \biggg).
\end{equation}
We recall the Kuznetsov formula~\cite{DeshouillersIwaniec1982,Kuznetsov1981,Proskurin1982}, which connects a sum over eigenvalues of the Laplacian weighted by the square of the $n$-th Fourier coefficient of the corresponding cusp form with a certain sum of Kloosterman sums. By exploiting this for a suitable test function (cf.~\cite{DeshouillersIwaniec1986}), Luo and Sarnak~\cite{LuoSarnak1995} gives
\begin{equation}\label{cancellation}
\sum_{t_{j} \leqslant T} x^{it_{j}} \ll T^{\frac{5}{4}} x^{\frac{1}{8}}(\log T)^{2}
\end{equation}
where $t_{j}$ ranges over the non-exceptional eigenvalues with respect to the modular group $SL(2, \Z)$. One easily shows that~\eqref{cancellation} also holds for typical arithmetic groups. Hence from partial summation, the ultimate form of the partial Euler product turns out to be
\begin{equation}\label{ultimate}
\zeta_{\Gamma, x}(s)
 = \varepsilon_{\Gamma}(s) \, \zeta_{\Gamma}(s) \exp \left(\Li(x^{1-s})+O(x^{\frac{7}{10}-\sigma} \log x) \right)
\end{equation}
on taking $T = x^{\frac{3}{10}}$. Here we have adopted the bound $\lambda_{1} \geqslant 21/100$ (\cite{Koyama1998,LuoRudnickSarnak1995}) where $\lambda_{1}$ is the lowest (non-zero) Laplace eigenvalue on $\Gamma \backslash \mathbb{H}$.
\end{proof}

\begin{remark}
Selberg~\cite[pp.506--520]{Selberg1989} conjectured that for congruence subgroups there are no exceptional eigenvalues, i.e. we have the bound $\lambda_{1} \geqslant 1/4$. For the full modular group, it is known that the lowest eigenvalue is quite large, $\lambda_{1} = 91.14 \dots$ according to numerical computations of Hejhal~\cite{Hejhal1976} and Steil~\cite{Steil1994}. It is ascertained by Huxley~\cite{Huxley1985} that for $q \leqslant 17$ the conjecture holds true. But as the level tends to infinity, we may find accumulation of eigenvalues of the principal congruence group $\Gamma(q)$ and the Hecke congruence group $\Gamma_{0}(q)$ arbitrary close to 1/4 (cf.~\cite{TakhtajanVinogradov1982}). In fact, one sees that
\begin{equation*}
\# \left\{j: \frac{1}{4} \leqslant \lambda_{j} < \frac{1}{4}+c \, (\log q)^{-2} \right\} 
\asymp |F|(\log q)^{-3}
\end{equation*}
for $\Gamma_{0}(q)$, where $c$ is a large constant (see~\cite{Iwaniec2002}, and~\cite{Huber1976} for cocompact settings). The best result so far established is
\begin{equation*}
\lambda_{1} \geqslant \frac{975}{4096} = \frac{1}{4}-\left(\frac{7}{64} \right)^{2}
\end{equation*}
according to Kim and Sarnak~\cite{KimSarnak2003} whose proof relies on advances on the functorial lifts $\mathrm{sym}^{3} \colon GL(2) \longrightarrow GL(4)$.
\end{remark}

The capital case $s = s_{k}$ with $1/2 < s_{k} \leqslant 1$ is given by
\begin{corollary}
For typical arithmetic $\Gamma$, we have
\begin{equation*}
\zeta_{\Gamma, x}(s_{k})
 =  e^{\gamma} \, \Res_{s = s_{k}} \zeta_{\Gamma}(s) \log x
\times \exp \bigggg(\sum_{\substack{\frac{1}{2} < s_{j} \leqslant 1 \\ j \ne k}} \Li(x^{s_{j}-s_{k}})
 - \frac{x^{\frac{1}{2}-s_{k}}}{\log x} \sum_{|t_{j}| \leqslant T} \frac{x^{it_{j}}}{s_{j}}
 + O \left(\frac{x^{1-s_{k}}}{T} \log x \right) \bigggg)
\end{equation*}
provided $1 \leqslant T \leqslant \sqrt{x}/\log x$, where $\gamma$ is Euler's constant.
\end{corollary}
\begin{proof}
Let
\begin{equation*}
\zeta_{\Gamma}(s) = \frac{\Res_{s = s_{k}} \zeta_{\Gamma}(s)}{s-s_{k}}+c_{0, k}(\Gamma)+O(|s-s_{k}|)
\end{equation*}
be the Laurent expansion around $s = s_{k}$ with $1/2 < s_{k} \leqslant 1$. The well-known asymptotic expansion~\cite[pp.126--131]{Berndt1994}
\begin{equation*}
\Li(x) = \gamma+\log |\log x|+\sum_{m \geqslant 1} \frac{(\log x)^{m}}{m! \, m}
\end{equation*}
provides us
\begin{equation*}
\zeta_{\Gamma}(s) \exp(\Li(x^{s_{k}-s}))
 = \Res_{s = s_{k}} \zeta_{\Gamma}(s) \, e^{\gamma} \log x+o_{s}(1) \qquad (s \to s_{k}).
\end{equation*}
Obviously we have $\zeta_{\Gamma}(s) \to \infty$ as $s \downarrow s_{k}$, so we can take $\varepsilon_{\Gamma}(s) = 1$ as $s \downarrow s_{k}$. By the approximation~\eqref{finalasymptotics}, the desired result follows immediately.
\end{proof}

At the point $s = 1$, by the correspondence described in Section~\ref{Iwaniec}, we have
\begin{equation}\label{unit}
\prod_{d \in D_{x}} (1-\varepsilon_{d}^{-2})^{-h(d)} \sim 2 \Res_{s = 1}(\zeta_{\Gamma}(s)) e^{\gamma} \log x,
\end{equation}
where $D_{x} = \{d \in D: \varepsilon_{d} \leqslant x \}$. This shall be deemed to be an analogue of Mertens' theorem. When finishing this paper, we recognized the work of Sharp~\cite{Sharp1991} who obtained~\eqref{unit} independently, from a different approach.

We refer to the work of Hashimoto, Iijima, Kurokawa and Wakayama~\cite{HashimotoIijimaKurokawaWakayama2004}, in which analogous constants of $\gamma$ are given. Especially it is investigated that, for $\Gamma$ a discrete cocompact torsion free subgroup of $SL(2, \R)$, the sum of reciprocals of squared Laplace eigenvalues is explicitly composed of ``Euler-Selberg constants.''

By optimizing a choice of $T$ in~\eqref{finalasymptotics}, we arrive at the following:
\begin{theorem}\label{equivalence}
Let $\Gamma$ be a typical arithmetic group. For $\alpha > 1/2$, the following are equivalent:
\begin{enumerate}
\item For $s \ne s_{j}$ with $1/2 < s_{j} \leqslant 1$ the limit
\begin{equation*}
\lim_{x \to \infty} \zeta_{\Gamma, x}(s) \prod \limits_{\frac{1}{2} < s_{j} \leqslant 1} \exp(-\Li(x^{s_{j}-s}))
\end{equation*}
exists and is non-zero on the half-plane $\Re s > \alpha$.
\item The prime geodesic theorem holds with $\mathcal{E}_{\Gamma}(x) \ll x^{\alpha}(\log x)^{\beta}$ for some $\beta > 0$.
\end{enumerate}
If $\mathcal{E}_{\Gamma}(x) = o(x^{\alpha} \log x)$, the condition 1 is valid on the line $\Re s = \alpha$ as well.
\end{theorem}
Recall that $\mathcal{E}_{\Gamma}(x)$ is the error term defined by~\eqref{pgt}.

We conclude this section by mentioning the behavior of $Z_{\Gamma, x}(s)$. We employ the connection between $\zeta_{\Gamma, x}(s, \rho)$ and $Z_{\Gamma, x}(s, \rho)$:
\begin{equation*}
Z_{\Gamma, x}(s, \rho) = \prod_{n \geqslant 0} \zeta_{\Gamma, x}(s+n, \rho)^{-1}.
\end{equation*}
\begin{corollary}
Let $\Gamma$ be a typical arithmetic group. We then have
\begin{description}
\item[Case I. $0 < \Re s < 1/2$] 
\begin{multline*}
\log Z_{\Gamma, x}(s)
 = -\Li(\Theta_{\Gamma}(x)^{1-s})-\frac{1}{2} \Li(\Theta_{\Gamma}(x)^{1-2s})-\dots-\frac{1}{n} \Li(\Theta_{\Gamma}(x)^{1-ns})\\
 - \sum_{\frac{1}{2} < s_{j} < 1} \sum_{1 \leqslant \nu \leqslant n} \frac{1}{\nu} \Li(x^{s_{j}-\nu s})
 + \frac{1}{\log x} \sum_{1 \leqslant \nu \leqslant n} \sum_{\frac{1}{2} < s_{j} < 1} \frac{x^{s_{j}-\nu s}}{\nu s_{j}}
 + O(x^{\frac{1}{2}-\sigma}(\log x)^{2})
\end{multline*}
with $n = \left[1+1/(2|\sigma|) \right]$;\\
\item[Case I\hspace{-0.1mm}I. $\Re s = 1/2$] 
\begin{equation*}
Z_{\Gamma, x}(s)^{-1} = \exp \biggg(\Li(\Theta_{\Gamma}(x)^{1-s})+\sum_{\frac{1}{2} < s_{j} < 1} \Li(x^{s_{j}-s})
 - \frac{x^{-s}}{\log x} \sum_{\frac{1}{2} < s_{j} < 1} \frac{x^{s_{j}}}{s_{j}}+O((\log x)^{2}) \biggg);
\end{equation*}
\item[Case I\hspace{-0.1mm}I\hspace{-0.1mm}I. $\Re s > 1/2$] 
\begin{multline*}
Z_{\Gamma, x}(s)^{-1}
 = \varepsilon_{\Gamma}(s) \, Z_{\Gamma}(s)^{-1} \\
\times \exp \biggg(\Li(\Theta_{\Gamma}(x)^{1-s})+\sum_{\frac{1}{2} < s_{j} < 1} \Li(x^{s_{j}-s})
 - \frac{x^{-s}}{\log x} \sum_{\frac{1}{2} < s_{j} < 1} \frac{x^{s_{j}}}{s_{j}}+O(x^{\frac{1}{2}-\sigma}(\log x)^{2}) \biggg).
\end{multline*}
\end{description}
\end{corollary}
On account of the positivity of $\prod_{n \geqslant 1}\zeta_{\Gamma}(\sigma+n)$ with $\sigma > 1/2$, the factor $\varepsilon_{\Gamma}(s)$ in Case I\hspace{-0.1mm}I\hspace{-0.1mm}I turns out to be the same as in Theorem~\ref{divergent}.

\section{Convergence of Euler products}\label{convEulerprod}%%%%%%%%%%%%%%%%%%%%%%%%%%%%%%%%
As explained in the introduction, we put effort into intertwining the \textit{weighted explicit formula} with the ordinary Dirichlet series for $\log \zeta_{\Gamma}(s, \rho)$. It appears to be intricate to apply Ramanujan's method to the situation for a general nontrivial $\rho$, since cancellation of the extra terms $\tr(\rho(p))$ must be taken into account. We cannot put forth any evidence to resolve this issue (for details, see Remark~\ref{nebentypus}). Fortunately, we could obtain the expression corresponding to the asymptotic formula~\eqref{finalasymptotics}. As expected, the explicit formula decisively affects the resulting behavior.

%For the sake of simplicity, let $\mathcal{C}(\Gamma \backslash \mathbb{H}, \rho)$ be the space of cusp forms and let $\mathcal{E}(\Gamma \backslash \mathbb{H}, \rho)$ be the space of incomplete Eisenstein series.
%Note that the orthogonal decomposition $\mathcal{L}^{2}(\Gamma \backslash \mathbb{H}, \rho) = \widetilde{\mathcal{C}}(\Gamma \backslash \mathbb{H}, \rho) \oplus \widetilde{\mathcal{E}}(\Gamma \backslash \mathbb{H}, \rho)$, where the tilde denotes closure in $\mathcal{L}^{2}(\Gamma \backslash \mathbb{H})$ with respect to the norm topology.
Let $\Delta$ denote the hyperbolic Laplacian acting on the space $L^{2}(\Gamma \backslash \mathbb{H}, \rho)$ of automorphic functions, square integrable over the fundamental domain $F = \Gamma \backslash \mathbb{H}$. For $\Gamma$ a congruence subgroup, Selberg~\cite{Selberg1956} has known that $\Delta$ has a point spectrum $\lambda_{0} \leqslant \lambda_{1} \leqslant \lambda_{2} \leqslant \dotsb$ with $\lambda_{j} \sim 4\pi j/|F|$ and thus, as one of the early triumphs of the trace formula, Maass cusp forms exist in abundance. The Laplacian $\Delta$ also has a continuous spectrum covering the interval $[1/4, \infty)$ with multiplicity equal to the number of inequivalent cusps of $F$.

We now summarize zeros and poles of the Selberg zeta function $Z_{\Gamma}(s, \rho)$ for a general cofinite $\Gamma$ (cf.~\cite{Venkov1982}). Denote by $\mathcal{R}$ a primitive elliptic conjugacy class in $\Gamma$ and let $d = d_{\mathcal{R}}$ be the order of $\mathcal{R}$. Any elliptic class having the same fixed points as those of $\mathcal{R}$ is of the form $\mathcal{R}^{j} \ (0 < j < d)$. As usual $k(\rho)$ stands for the total degree of singularity of $\rho$. We anomalistically count $s_{0}$ as the exceptional zero, since it is inherently exceptional for nontrivial $\rho$.
\begin{enumerate}
\item [] \textbf{Trivial zeros}\\
\begin{enumerate}
\item [$\bullet$] $s = -n \ (n = 0, 1, 2...)$ with multiplicity of
\begin{equation*}
\dfrac{|F|}{\pi} \dim \rho \left(n+\frac{1}{2} \right)
 - \sum \limits_{\mathcal{R}} \sum \limits_{k=1}^{d-1} 
\dfrac{\tr(\rho(\mathcal{R})^{k})}{d \sin \frac{k \pi}{d}} \sin \dfrac{k(2n+1) \pi }{d};
\end{equation*}
\end{enumerate}
\item [] \textbf{Trivial poles}\\
\begin{enumerate}
\item [$\bullet$] $s = 1/2$ with multiplicity $(k(\rho)-\tr \, \Phi(1/2, \rho))/2$ where $\Phi$ is the scattering matrix;\\
\item [$\bullet$] $s = 1/2-l \ (l = 1, 2, 3,...)$ with multiplicity $k(\rho)$;\\
\end{enumerate}
\item [] \textbf{Nontrivial zeros}\\
\begin{enumerate}
\item [$\bullet$] the exceptional zeros $s = s_{j} =1/2 \pm it_{j} \in [0, 1]$ where $\lambda_{j} = s_{j}(1-s_{j})$ are the exceptional part of a pure point spectrum of the Laplacian on the cuspidal subspace $L^{2}_{\mathrm{cusp}}(\Gamma \backslash \mathbb{H}, \rho)$, the so-called \textit{cuspidal} eigenvalues, and the multiplicity is equal to that of $\lambda_{j}$;\\
\item [$\bullet$] the exceptional zeros $s = s_{j} \in (1/2, 1]$ which are the poles of $\varphi(s, \rho)$--the constant term in the Fourier expansion of the Eisenstein series--in the segment $(1/2, 1]$, and yield the so-called \textit{residual} eigenvalues $\lambda_{j} = s_{j}(1-s_{j})$;\\
\item [$\bullet$] the non-exceptional zeros $s = s_{j} = 1/2 \pm it_{j}$ on the line $\Re s = 1/2$ where $\lambda_{j} = s_{j}(1-s_{j}) \geqslant 1/4$ are eigenvalues of the Laplacian corresponding to cusp forms, the multiplicity of $s_{j}$ is equal to that of $\lambda_{j}$;\\
\item [$\bullet$] $s = \rho_{j}$ where $\rho_{j} = \beta_{j}+i \gamma_{j}$ are poles of $\varphi(s, \rho)$ which lie on the half-plane~$\beta_{j} < 1/2$.
\end{enumerate}
\end{enumerate}
Recall that the exceptional zeros in $(1/2, 1]$ are ordered as $s_{0} \geqslant s_{1} \geqslant \dotsb$. Since the lowest eigenvalue $\lambda_{0} = 0$ exists iff $\rho = \1$, $Z_{\Gamma}(s)$ has a simple zero at $s = s_{0} = 1$ which corresponds to the constant eigenfunction $u_{0}(z) = |F|^{-\frac{1}{2}}$. 
%It should be pointed out that the final picture of the zeros and poles emerges from compiling the descriptions above. 
In the case of congruence subgroups, we see that $\varphi(s, \rho)$ has no poles in $(1/2, 1)$ so that in this case there is no residual spectrum besides $\lambda_{0} = 0$. Then the pure point spectrum $\lambda_{j} \ne 0$ always accounts for the cuspidal one.

Consequently, the zeros and poles of $\zeta_{\Gamma}(s, \rho)$ are redundantly described as follows:
\begin{enumerate}
\item [] \textbf{Trivial zeros}\\
\begin{enumerate}
\item [$\bullet$] $s = 1/2$ with multiplicity $(k(\rho)-\tr \, \Phi(1/2, \rho))/2$;\\
\item [$\bullet$] $s = -1/2$ with multiplicity $(k(\rho)+\tr \, \Phi(1/2, \rho))/2$;\\
\end{enumerate}
\item [] \textbf{Trivial poles}\\
\begin{enumerate}
\item [$\bullet$] $s = -n \ (n = 0, 1, 2, 3,...)$ with multiplicity of
\begin{align*}
	\begin{cases}
		\dfrac{|F|}{2\pi} \dim \rho
		 - \sum \limits_{\mathcal{R}} \sum \limits_{k=1}^{d-1} \dfrac{\tr(\rho(\mathcal{R})^{k})}{d} & n = 0\\[10pt]
		\dfrac{|F|}{\pi} \dim \rho
		 - 2 \sum \limits_{\mathcal{R}} \sum \limits_{k=1}^{d-1} 
		\dfrac{\tr(\rho(\mathcal{R})^{k})}{d} \cos \dfrac{2\pi kn}{d} & \text{otherwise}
	\end{cases};
\end{align*}
\end{enumerate}
\item [] \textbf{Nontrivial zeros}\\
\begin{enumerate}
\item [$\bullet$] $s = s_{j}-1 = -1/2 \pm it_{j} \in [-1, 0]$ with multiplicity equal to that of the corresponding eigenvalue $\lambda_{j}$;\\
\item [$\bullet$] $s = s_{j}-1 = -1/2 \pm it_{j}$ with multiplicity equal to that of $\lambda_{j}$;\\
\item [$\bullet$] $s = \rho_{j}-1$ on the half-plane $\Re s < -1/2$;\\
\end{enumerate}
\item [] \textbf{Nontrivial poles}\\
\begin{enumerate}
\item [$\bullet$] the exceptional zeros $s = s_{j} = 1/2 \pm \sqrt{1/4-\lambda_{j}} \in [0, 1]$ with multiplicity equal to that of $\lambda_{j}$;\\
\item [$\bullet$] the non-exceptional zeros $s = s_{j} = 1/2 \pm i\sqrt{\lambda_{j}-1/4}$ with multiplicity equal to that of $\lambda_{j}$;\\
\item [$\bullet$] $s = \rho_{j}$ on the half-plane $\Re \rho_{j} < 1/2$.
\end{enumerate}
\end{enumerate}

We define for the sake of simplicity the \textit{weighted psi function}
\begin{equation*}
\Psi_{\Gamma}^{\mathrm{w}}(s, x, \rho) \coloneqq \sum_{N(p)^{k} \leqslant x} \frac{\tr(\rho(p)^{k}) \log N(p)}{N(p)^{ks}}.
\end{equation*}
The following explicit formula dominates the proof of Theorem~\ref{convergent} (see below):
\begin{lemma}\label{weighted}
Let $\Gamma$ be a typical arithmetic group.
For a fixed $s \in \C$ such that $\frac{1}{2} < \Re s \leqslant 1, \, s \ne \forall s_{j}$ we have
\begin{equation*}
\Psi_{\Gamma}^{\mathrm{w}}(s, x, \rho)
 = -\dfrac{\zeta_{\Gamma}^{\prime}}{\zeta_{\Gamma}}(s, \rho)+x^{-s} \Psi_{\Gamma}(x, \rho)
 + \sum_{\frac{1}{2} < s_{j} \leqslant 1} \frac{s x^{s_{j}-s}}{s_{j}(s_{j}-s)}+O(x^{\frac{1}{2}-\sigma}(\log x)^{3}).
\end{equation*}
\end{lemma}

\begin{proof}
It is easy to see that for $\alpha = 1+\eta$ and $\delta \in \C$, Perron's formula (cf.~\cite{Davenport2000})~gives
\begin{multline}\label{perron}
 - \frac{1}{2\pi i} \int_{\alpha-iT}^{\alpha+iT} \frac{\zeta_{\Gamma}^{\prime}}{\zeta_{\Gamma}}(z, \rho) \frac{x^{z}}{z-\delta} \mathrm{d} z
 =  x^{\delta} \sideset{}{'} \sum_{N(p)^{k} \leqslant x} \frac{\tr(\rho(p)^{k}) \log N(p)}{N(p)^{k \delta}}\\
 + O \left(x^{\alpha} \sum_{p, k} 
\min \left\{1, \left(T \left|\log \frac{x}{N(p)^{k}} \right| \right)^{-1} \right\} \frac{\tr(\rho(p)^{k}) \log N(p)}{N(p)^{k \alpha}} \right),
\end{multline}
where the prime on the sum denotes that the last term of the sum is weighted by 1/2 only when $x = N(p)^{k}$ for some $k \geqslant 1$. A simple computation and a Brun-Titchmarsh type inequality (Theorem~\ref{Bykovskii}) show that the error term can be reduced as
\begin{equation*}
O \left(\frac{x^{1+\eta}}{\eta(\rho) T}+\frac{x(\log x)^{2}}{T}+\sqrt{x}(\log x)^{3} \right),
\end{equation*}
where $\eta(\rho) = 1$ iff $\rho \ne \1$ and otherwise $\eta(\rho) = \eta$.
Subtracting the identity~\eqref{perron} with $\delta = 0$ from that with $\delta = s$ leads one to
\begin{multline}\label{weighted_explicit_formula}
 - \frac{1}{2\pi i} \int_{\alpha-iT}^{\alpha+iT} 
\frac{\zeta_{\Gamma}^{\prime}}{\zeta_{\Gamma}}(z, \rho) \frac{s x^{z-s}}{z(z-s)} \mathrm{d} z\\
 = \sum_{N(p)^{k} \leqslant x} \frac{\tr(\rho(p)^{k}) \log N(p)}{N(p)^{ks}}
 - x^{-s} \sum_{N(p)^{k} \leqslant x} \tr(\rho(p)^{k}) \log N(p)\\
 + O \left(\frac{x^{1+\eta-\sigma}}{\eta(\rho) T}+\frac{x^{1-\sigma}(\log x)^{2}}{T}+x^{\frac{1}{2}-\sigma}(\log x)^{3} \right).
\end{multline}

According to Iwaniec's method~\cite{Iwaniec1984} we move the contour into the five segments:
\begin{align*}
&\mathscr{L}_{1}^{\pm} = \biggl[\frac{1}{2}+\eta \pm iT, \, 1+\eta \pm iT \biggr],\\
&\mathscr{L}_{2}^{\pm} = \biggl[-\eta \pm iT, \, \frac{1}{2}+\eta \pm iT \biggr]
\quad \text{and} \quad \mathscr{L}_{3} = [-\eta - iT, \, -\eta + iT],
\end{align*}
where $\eta = (\log T)^{-1}$.

The poles of the integrand at $z = s$ and $z = s_{j}$ contribute
\begin{equation}\label{residues}
 - \frac{\zeta_{\Gamma}^{\prime}}{\zeta_{\Gamma}}(s, \rho)
 + \sum_{\frac{1}{2} < s_{j} \leqslant 1} \frac{s x^{s_{j}-s}}{s_{j}(s_{j}-s)}+O(x^{\frac{1}{2}-s} \log T)
\end{equation}
to the integral, where the error term is deduced from the obvious evaluation of the sum of the non-exceptional poles of $\zeta_{\Gamma}(s, \rho)$ without considering any cancellation. Contributions from all other poles, for instance the one from the determinant of the scattering matrix $\varphi(s, \rho)$ (cf.~\cite{LaxPhillips1976}~\cite{Shahidi1990}), are absorbed by the error term.

It remains to execute the integrals along the segments $\mathscr{L}_{1}^{\pm}$, $\mathscr{L}_{2}^{\pm}$ and $\mathscr{L}_{3}$. It is achieved by employing the method in~\cite{Iwaniec1984}:
\begin{align}\label{segment1}
\left|\frac{1}{2\pi i} \int_{\mathscr{L}_{1}^{\pm}} \frac{\zeta_{\Gamma}^{\prime}}{\zeta_{\Gamma}}(z, \rho) 
\frac{s x^{z-s}}{z(z-s)} \mathrm{d} z \right|
& \ll \frac{1}{T} \left(\sqrt{x}+\frac{x}{T} \right) x^{\eta-\sigma} \log T,\\
\left|\frac{1}{2\pi i} \int_{\mathscr{L}_{2}^{\pm}} \frac{\zeta_{\Gamma}^{\prime}}{\zeta_{\Gamma}}(z, \rho) \label{segment2}
\frac{s x^{z-s}}{z(z-s)} \mathrm{d} z \right|
& \ll \frac{x^{\frac{1}{2}+\eta-\sigma}}{T} \log T,\\
\left|\frac{1}{2\pi i} \int_{\mathscr{L}_{3}} \frac{\zeta_{\Gamma}^{\prime}}{\zeta_{\Gamma}}(z, \rho) \label{segment3}
\frac{s x^{z-s}}{z(z-s)} \mathrm{d} z \right|
& \ll x^{-\eta-\sigma} \log T.
\end{align}
The best we can choose is $T = \sqrt{x}/\log x$. We then insert~\eqref{residues} and the above evaluations into the underlying formula~\eqref{weighted_explicit_formula}. This completes the proof.
\end{proof}

We additionally provide an elementary lemma which is an analogue of Conrad's theorem~\cite[Theorem 3.3]{Conrad2005}. This drives the proof of the main theorem below.
\begin{lemma}\label{Conrad}
Let $\Gamma$ be a cofinite subgroup of $SL(2, \R)$ and $\rho$ any unitary representation. For $z \in \C$ with $\Re z > 1/2$, the following are equivalent:
\begin{enumerate}
\item The limit
\begin{equation*}
\lim_{x \to \infty} \zeta_{\Gamma, x}(z, \rho) \prod \limits_{\frac{1}{2} < s_{j} \leqslant 1} \exp(-\Li(x^{s_{j}-z}))
%\prod_{N(p) \leqslant x} \det \left(I_{\dim \rho}-\tr(\rho(p))N(p)^{-z} \right)^{-1}
\end{equation*}
is nonzero.\\
\item The limit
\begin{equation}\label{limit}
\lim_{x \to \infty} \biggg(\sum_{N(p)^{k} \leqslant x} \frac{\tr(\rho(p)^{k})}{k N(p)^{kz}}
 - \sum_{\frac{1}{2} < s_{j} \leqslant 1} \Li(x^{s_{j}-z}) \biggg)
\end{equation}
exists.
\end{enumerate}
If $z$ satisfies the condition 1 or 2, the series $\log \zeta_{\Gamma}(s, \rho)-\sum_{\frac{1}{2} < s_{j} \leqslant 1} \Li(x^{s_{j}-s})$ converges for $s = z$ and for $\Re s > \Re z$. Also, the limit
\begin{equation}\label{limit2}
\lim_{x \to \infty} \zeta_{\Gamma, x}(s, \rho) \prod \limits_{\frac{1}{2} < s_{j} \leqslant 1} \exp(-\Li(x^{s_{j}-s}))
\end{equation}
agrees with $\varepsilon_{\Gamma}(s, \rho) \, \zeta_{\Gamma}(s, \rho)$ at $s = z$ and in $\Re s > \Re z$ with~$\varepsilon(s, \rho) = \pm 1$.
\end{lemma}
This says that if the limit of the product~\eqref{limit2} is nonzero, its value is equal to the value $\varepsilon_{\Gamma}(s, \rho) \, \zeta_{\Gamma}(s, \rho)$. The proof is based upon Abel's theorem but the complete proof is omitted here (almost just the same as~\cite{Conrad2005}). Like in the case of $\rho = \1$, it seems to be plausible enough that, recalling Remark~\ref{epsilonremark},
\begin{equation}\label{epsilon2}
\varepsilon_{\Gamma}(s, \rho) = 
\begin{cases}
	-1 & \text{$\zeta_{\Gamma}(\sigma, \rho)$ is real and nonpositive}\\
	1 & \text{otherwise}
	\end{cases}.
\end{equation}

Now we are ready to prove the following theorem:
\begin{theorem}\label{convergent}
Let $\Gamma$ be a typical arithmetic group and $\rho$ any unitary representation (not necessarily nontrivial). For $\Re s > 1/2$ we then have
\begin{equation*}
\zeta_{\Gamma, x}(s, \rho)
 = \varepsilon_{\Gamma}(s, \rho) \, \zeta_{\Gamma}(s, \rho)
\times \exp \biggg(\sum_{\frac{1}{2} < s_{j} \leqslant 1} \Li(x^{s_{j}-s})
 + \frac{x^{\frac{1}{2}-s}}{\log x} \sum_{|t_{j}| \leqslant T}\frac{x^{it_{j}}}{s_{j}}+O \left(\frac{x^{1-\sigma}}{T} \log x \right) \biggg)
\end{equation*}
provided $1 \leqslant T \leqslant \sqrt{x}/\log x$, where $\varepsilon(s, \rho) = \pm 1$.
\end{theorem}

\begin{proof}
Our goal is to clarify the behavior of the following series:
\begin{equation*}
\log \zeta_{\Gamma, x}(s, \rho)
 = \sum_{N(p) \leqslant x} \sum_{k = 1}^{\infty} \frac{\tr(\rho(p)^{k})}{k N(p)^{ks}}.
\end{equation*}
Since for $\Re s > 1/2$ we have
\begin{equation*}
\sum_{k = 2}^{\infty} \sum_{\sqrt[k]{x} < N(p) \leqslant x} \frac{\tr(\rho(p)^{k})}{k N(p)^{ks}} \to 0
\qquad \text{as} \qquad x \to \infty,
\end{equation*}
it suffices to consider the sum
\begin{equation*}
\sum_{N(p)^{k} \leqslant x} \frac{\tr(\rho(p)^{k})}{k N(p)^{ks}}.
\end{equation*}
%where we put $c = p^{m}$ to be a hyperbolic conjugacy class which is not necessarily primitive.
From summation by parts, there follows
\begin{equation}\label{logzeta}
\sum_{N(p)^{k} \leqslant x} \frac{\tr(\rho(p)^{k})}{k N(p)^{ks}}
 = \mathrm{const}+\frac{\Psi_{\Gamma}^{\mathrm{w}}(s, x, \rho)}{\log x}
 - \int_{N(p_{1})}^{x} \Psi_{\Gamma}^{\mathrm{w}}(s, t, \rho) \mathrm{d} \left(\frac{1}{\log t} \right)
\end{equation}
where, as in the previous section, $N(p_{1})$ stands for the lowest norm for each $\Gamma$ and ``const'' depends solely on $s, \, \Gamma$ and $\rho$.
By appealing to Lemmas~\ref{slight} and~\ref{weighted},
%\begin{equation*}
%\Psi_{\Gamma}(x, \rho) = \sum_{\frac{1}{2} < s_{j} \leqslant 1} \frac{x^{s_{j}}}{s_{j}}
% + \sqrt{x} \sum_{|t_{j}| \leqslant T} \frac{x^{ir_{j}}}{s_{j}}+O \left(\frac{x}{T}(\log x)^{2} \right)
%\quad (1 \leqslant T \leqslant \sqrt{x}/\log x),
%\end{equation*}
the left hand side of the equation~\eqref{logzeta} turns out to be
\begin{equation}\label{ultimate2}
%\mathrm{const}+\frac{\Psi_{\Gamma}(x, \rho)}{x^{s} \log x}
% + \frac{1}{\log x} \sum_{\frac{1}{2} < s_{j} \leqslant 1} \frac{s x^{s_{j}-s}}{s_{j}(s_{j}-s)}
% + \sum_{|t_{j}| \leqslant T} \frac{s x^{s_{j}-s}}{s_{j}(s_{j}-s)}\\
% + \int \frac{\Psi_{\Gamma}(x, \rho)}{x^{1+s}(\log x)^{2}} \mathrm{d} x
% + \sum_{\frac{1}{2} < s_{j} \leqslant 1} \frac{s}{s_{j}(s_{j}-s)} \int \frac{x^{s_{j}-s}}{x(\log x)^{2}} \mathrm{d} x
% + O(x^{\frac{1}{2}-\sigma}(\log x)^{3})\\
% = 
\mathrm{const}+\sum_{\frac{1}{2} < s_{j} \leqslant 1} \Li(x^{s_{j}-s})
 + \sum_{|t_{j}| \leqslant T} \left(1-\frac{s}{s_{j}} \right) \Li(x^{s_{j}-s})
 + O \left(\frac{x^{1-\sigma}}{T} \log x \right).
\end{equation}
Clearly the latter sum is
\begin{equation*}
\frac{x^{\frac{1}{2}-s}}{\log x} \sum_{|t_{j}| \leqslant T} \frac{x^{it_{j}}}{s_{j}}
 + O(\sqrt{x}),
\end{equation*}
by recalling again that there are $O(T^{2})$ eigenvalues with $t_{j} \leqslant T$. When the error term is $o(1)$, it is justified that $\mathrm{const} = \log(\varepsilon_{\Gamma}(s, \rho) \, \zeta_{\Gamma}(s, \rho))$ via Lemma~\ref{Conrad}. Hence we have completed the proof of the theorem.
\end{proof}
Note that the sum over $|t_{j}| \leqslant T$ in Theorem~\ref{convergent} is refined by the formula
\begin{equation*}
\frac{x^{\frac{1}{2}-s}}{\log x} \sum_{|t_{j}| \leqslant T} \frac{x^{it_{j}}}{s_{j}} = \sum_{|t_{j}| \leqslant T} \Li(x^{s_{j}-s})+O(\sqrt{x}).
\end{equation*}

As an immediate consequence, we specify $T \coloneqq x^{\frac{1}{4}} \log x$, obtaining
\begin{corollary}\label{3/4}
Let $\Gamma$ be a typical arithmetic group and $\rho$ any unitary representation. For $\Re s \geqslant 3/4$ and $s \ne s_{j}$ with $1/2 < s_{j} \leqslant 1$ we then have
\begin{equation}\label{limitingvalue}
\lim \limits_{x \to \infty} \zeta_{\Gamma, x}(s, \rho) \prod \limits_{\frac{3}{4} < s_{j} \leqslant 1} \exp(-\Li(x^{s_{j}-s}))
 = \varepsilon_{\Gamma}(s, \rho) \, \zeta_{\Gamma}(s, \rho).
\end{equation}
\end{corollary}
Undoubtedly the above argument is applicable to the case that $\rho = \1$ in which one can replace the exponent $3/4$ by $7/10$ via cancellation of terms in the spectral exponential sum $\sum_{t_{j} \leqslant T} x^{it_{j}}$. Notice that the asymptotic Theorem~\ref{Bykovskii} permits us to replace $\sqrt{x}(\log x)^{4}$ with $\sqrt{x}(\log x)^{3}$, yet at this stage this slight improvement is not of importance for extension of the conditional convergence region of the limit in Corollary~\ref{3/4}.\\

We now focus on Weil's estimate for the Kloosterman sums $\mathcal{S}_{\chi}(m, n; c)$ associated with a certain multiplier $\chi$ to obtain Corollary~\ref{3/4} in a wider region. There is no point in making the above theorem as general as possible, and the typical arithmetic case is sufficiently hazy and attractive to work with. Thus from now on we suppose that $\Gamma$ is the Hecke congruence group $\Gamma_{0}(q)$. Let $\chi$ be an odd Dirichlet character to the modulus $q \in \N$. The nebentypus $\chi$ induces a one-dimensional unitary representation of $\Gamma$ by the consistency condition
\begin{equation*}
\rho(\gamma) = \chi(d), \qquad 
\text{if} \qquad \gamma =  \begin{pmatrix} a & b \\ c & d \end{pmatrix} \in \Gamma.
\end{equation*}
Define the Kloosterman sum (cf.~\cite{Davenport1933,Weil1948})
\begin{equation*}
\mathcal{S}_{\chi}(m, n; c)
 = \sum_{\substack{d \tpmod{c} \\ (d, c) = 1}} \chi(d) e \left(\frac{m \overline{d}+nd}{c} \right), 
\qquad \text{where} \qquad e(t) = e^{2\pi it}
\end{equation*}
and individual one satisfies Weil's bound
\begin{equation*}
|\mathcal{S}_{\chi}(m, n; c)| \leqslant (m, n, c)^{\frac{1}{2}} c^{\frac{1}{2}} \tau(c).
\end{equation*}
Here $(m, n, c)$ denotes the greatest common divisor and $\tau(c)$ the number of divisors of $c$. This bound is the best possible and have been frequently used in analytic number theory. It implies that the Selberg-Kloosterman zeta function (cf.~\cite{Linnik1962})
\begin{equation*}
Z_{m, n}(s, \chi) = \sum_{c \equiv 0 \tpmod{q}} \frac{\mathcal{S}_{\chi}(m, n; c)}{c^{2s}}
\end{equation*}
converges absolutely in $\Re s > 3/4$ for suitable $m, \, n$, whence Selberg~\cite{Selberg1965} succeeded in showing that the lowest eigenvalue satisfies $\lambda_{\delta(\chi)} \geqslant 3/16$ (cf.~\cite{Iwaniec1989}), where
\begin{equation*}
\delta(\chi) = 
	\begin{cases}
	1 & \text{if $\chi = \1$}\\
	0 & \text{otherwise}
	\end{cases}.
\end{equation*}
Hence Selberg appealed indirectly to the RH for curves over finite fields. Later, Gelbart and Jacquet~\cite{GelbartJacquet1978} established the strict inequality $\lambda_{\delta(\chi)} > 3/16$ by means of the symmetric square functorial lifts $\mathrm{sym}^{2} \colon GL(2) \longrightarrow GL(3)$. When we stop the discussion here, we could deduce~\eqref{limitingvalue} in $\Re s \geqslant 3/4$ again. Although there are not so many researches treating Selberg's conjecture attached with a representation $\rho$, Iwaniec~\cite[Corollary]{Iwaniec1990} gave the bound $\lambda_{1} \geqslant 44/225$ with the quadratic character $\chi_{p}(d) = \left(\frac{d}{p} \right)$ for almost all groups $\Gamma_{0}(p)$ with $p$ a prime. Observe $3/4 > 11/15$ and that a cancellation theorem attached with $\chi$ holds alike, and we have~\eqref{limitingvalue} in $\Re s > 11/15$ for such groups. For more information of the Kloosterman sums, see~\cite{Iwaniec1997}. The following states the connection between the Selberg zeta function and the Selberg-Kloosterman zeta function:
\begin{theorem}\label{SelbergKloosterman}
Let $\Gamma = \Gamma_{0}(q)$ and $\chi$ a nebentypus. Suppose that $\Theta_{\Gamma}(x) = x+O(x^{\frac{1}{2}+\varepsilon})$ for $\chi = \1$ and otherwise $\Theta_{\Gamma}(x, \chi) = O(x^{\frac{1}{2}+\varepsilon})$. We then have
\begin{enumerate}
\item[---] For $\chi = \1$, the region of holomorphic continuation of $Z_{m, n}(s)$ agrees with that of conditional convergence of $\zeta_{\Gamma, x}(s) \exp(-\Li(x^{1-s}))$ as $x \to \infty$.\\
\item[---] For $\chi \ne \1$, the region of holomorphic continuation of $Z_{m, n}(s, \chi)$ agrees with that of conditional convergence of $\zeta_{\Gamma, x}(s, \chi)$ as $x \to \infty$.
\end{enumerate}
\end{theorem}

\begin{proof}
For simplicity we take $\chi = \1$. The asymptotic formula~\eqref{finalasymptotics} and the hypothesis $\Theta_{\Gamma}(x) = x+O(x^{\frac{1}{2}+\varepsilon})$ give
\begin{equation*}
\zeta_{\Gamma, x}(s) \exp(-\Li(x^{1-s})) = \varepsilon_{\Gamma}(s) \, \zeta_{\Gamma}(s) 
\exp \biggg(\sum_{\frac{1}{2} < s_{j} < 1} \Li(x^{s_{j}-s})+O(x^{\frac{1}{2}-\sigma+\varepsilon}) \biggg).
\end{equation*}
On the other hand, on account of the functional equation~\cite[Theorem 9.2]{Iwaniec2002}, the poles of the zeta function $Z_{m, n}(s)$ in the half-plane $\Re s > 1/2$ are at the points $s_{j}$ on the segment $(1/2, 1)$. Hence extending the holomorphic region of $Z_{m, n}(s)$ is equivalent to proving non-existence of the exceptional eigenvalues within a certain range. From the above consideration we complete the proof of Theorem~\ref{SelbergKloosterman} for $\chi = \1$. For $\chi \ne \1$ the argument can be progressed like in the case of $\chi = \1$.
\end{proof}

It would be difficult to generalize Theorem~\ref{SelbergKloosterman} to the case with a finite-dimensional representation. Although $Z_{\Gamma}(s, \chi)$ has no natural development into decent Dirichlet series, the theorem reveals that extending the convergence region of $Z_{\Gamma}(s, \chi)$ is reduced to showing the possible holomorphic region of the Dirichlet series $Z_{m, n}(s, \chi)$. We particularly stress that the above hypothesis $\Theta_{\Gamma}(x) = x+O(x^{\frac{1}{2}+\varepsilon})$ when $\chi = \1$ follows from the conjecture~\cite[Conjecture 2.2]{PetridisRisager2017} concerning cancellation of terms in the spectral exponential sum.

In view of a choice of $T$ in Theorem~\ref{convergent}, the following is easily deduced: 
\begin{theorem}\label{equivalence2}
Let $\Gamma$ be a typical arithmetic group. For $\alpha > 1/2$ and any unitary representation $\rho$, the following are equivalent:
\begin{enumerate}
\item For $s \ne s_{j}$ with $1/2 < s_{j} \leqslant 1$ the limit
\begin{equation*}
\lim_{x \to \infty} \zeta_{\Gamma, x}(s, \rho) \prod \limits_{\frac{1}{2} < s_{j} \leqslant 1} \exp(-\Li(x^{s_{j}-s}))
\end{equation*}
is non-zero on the half-plane $\Re s > \alpha$.
\item The prime geodesic theorem~\eqref{pgt} holds with $\mathcal{E}_{\Gamma}(x, \rho) \ll x^{\alpha}(\log x)^{\beta}$ for some $\beta > 0$.
\end{enumerate}
If $\mathcal{E}_{\Gamma}(x, \rho) = o(x^{\alpha} \log x)$, the condition 1 is valid on the line $\Re s = \alpha$ as well.
\end{theorem}
The proof of this is the same as that of Theorem~\ref{equivalence} and this completely encompasses Theorem~\ref{equivalence}. By Lemma~\ref{Conrad} the condition 1 is also equivalent to
\begin{equation}\label{lim}
\lim_{x \to \infty} \zeta_{\Gamma, x}(s, \rho) \prod_{\frac{1}{2} < s_{j} \leqslant 1} \exp(-\Li(x^{s_{j}-s}))
 = \varepsilon_{\Gamma}(s, \rho) \, \zeta_{\Gamma}(s, \rho)
\end{equation}
on the half-plane $\Re s > \alpha$. By~\eqref{lim} we find that the condition 1 implies that
\begin{equation}\label{Akatsuka}
\displaystyle{\lim \limits_{x \to \infty}}
\frac{\prod \limits_{N(p) \leqslant x} \det \left(I_{\dim \rho}-\rho(p) N(p)^{-s} \right)^{-1}}
{\prod \limits_{\frac{1}{2} < s_{j} \leqslant 1} \exp \left[\lim \limits_{\varepsilon \downarrow 0}
\left(\displaystyle{\int_{1+\varepsilon}^{x}} \frac{u^{s_{j}-s}}{\log{u}} \frac{\mathrm{d} u}{u}
 - \frac{x^{s_{j}-s}}{s_{j} \log x}-\log{\frac{1}{\varepsilon}} \right) \right]}
 = \varepsilon_{\Gamma}(s, \rho) \, \zeta_{\Gamma}(s, \rho) \prod_{\frac{1}{2} < s_{j} \leqslant 1} e^{\gamma}(s-s_{j})
\end{equation}
in $\Re s > \alpha$, which superficially mimics Akatsuka's complicated form~\cite{Akatsuka2017}.

As a simple corollary of Theorem~\ref{equivalence} we get the following:
\begin{corollary}
Assume Selberg's eigenvalue conjecture and let $\rho$ be nontrivial. If $\mathcal{E}_{\Gamma}(x, \rho) \ll x^{\alpha}(\log x)^{\beta}$ for some $\alpha, \beta > 0$, the limit $\lim_{x \to \infty} \zeta_{\Gamma, x}(s, \rho)$ is nonzero in $\Re s > \alpha$.
\end{corollary}

\begin{remark}\label{nebentypus}
For some special choices of $\rho$, we could utilize Ramanujan's method in place of the method formulated in this section (see~\cite{Kaneko2018}). For example, It is likely that his method works effectively as well for the Hecke congruence group $\Gamma_{0}(q)$ attached with a character $\chi$.
%In his method, the flexibility on representations $\rho$ is not always justified.
%And of course the flexibility also depends on forms of Euler products; it gets harder to perform the method as complexity of the products increases (e.g. the form such as~\eqref{ellipticcurveL}).
\end{remark}

\end{document}